\documentclass[]{siamart220329}

\usepackage{amsfonts}
\usepackage{graphicx}
\usepackage{epstopdf}
\usepackage{algorithm}
\usepackage{algorithmicx}
\usepackage{algpseudocode}
\usepackage{esint}
\usepackage{graphicx}

\usepackage{tikz}
\usetikzlibrary{decorations.pathreplacing}

\ifpdf
  \DeclareGraphicsExtensions{.eps,.pdf,.png,.jpg}
\else
  \DeclareGraphicsExtensions{.eps}
\fi

\newsiamremark{remark}{Remark}
\newsiamremark{hypothesis}{Hypothesis}
\crefname{hypothesis}{Hypothesis}{Hypotheses}
\newsiamthm{claim}{Claim}

\headers{A-posteriori error estimates for systems of hyperbolic CL}{J. Giesselmann, A. Sikstel}

\title{A-posteriori error estimates for systems of hyperbolic conservation laws\thanks{Submitted to the editors \today.
\funding{ This work was funded by DFG in the
 Collaborative Research Centre
CRC/Transregio 154,
Mathematical Modelling, Simulation and Optimization Using the Example of Gas Networks,
Project  C05.}}}

\author{Jan Giesselmann\thanks{Technical University of Darmstadt, Department of Mathematics, Dolivostr. 15, 64293 Darmstadt,
  (\email{ giesselmann@mathematik.tu-darmstadt.de}).}
\and Aleksey Sikstel\thanks{University of Cologne, Department of Mathematics, Weyertal 86 – 90, 50931 Köln, Germany, \email{a.sikstel@uni-koeln.de}).}
}

\usepackage{amsopn}

\renewcommand{\vec}[1]{\boldsymbol{#1}}

\newcommand\R{\mathbb{R}}

\newcommand{\entropy}{e}
\newcommand{\entropyFlux}{q}
\newcommand{\consDom}{\mathbb{D}}
\newcommand{\dom}{\Omega}
\newcommand{\stDom}{\dom_T}
\newcommand{\pwConst}{\mathbf{P}_0}
\newcommand{\basisSub}{\mathbf{V}}
\newcommand{\basisSubPlus}{\basisSub^+}
\newcommand{\edges}[1]{\mathbf{E}\left(#1 \right)}
\newcommand{\edgess}[1]{\mathbf{E}'\left(#1 \right)}

\newcommand{\opB}{\mathcal{B}}
\newcommand{\opTestToFem}{\mathcal{P}}
\newcommand{\opE}{\mathcal{E}}
\newcommand{\opSemiG}{\mathcal{S}}
\newcommand{\node}{\nu}
\newcommand{\nodeSet}{\mathbf{N}}
\newcommand{\num}[1]{\widehat{#1}}
\newcommand{\numflux}{\widehat{\vec{f}}}
\newcommand{\tfct}{\varphi}
\newcommand{\ttfct}{\phi}
\newcommand{\tttfct}{\psi}

\newcommand{\av}[1]{\overline{#1}}
\newcommand{\avPhi}{\av{\varphi}}
\newcommand{\avPsi}{\av{\tttfct}}

\newcommand{\shockThresh}{\sigma}
\newcommand{\trapS}{S}
\newcommand{\trapG}{G}
\newcommand{\macroT}{\tau}
\newcommand{\osc}{\vec{\kappa}}

\newcommand{\locBscalar}[0]{\beta}
\newcommand{\locEscalar}[0]{\eta}

\newcommand{\estim}[0]{\vec{\mathfrak{E}}}

\newcommand{\diff}{\mathrm{D}}

\begin{document}
\maketitle

\begin{abstract}
  We provide rigorous and computable a-posteriori error estimates for first order finite-volume approximations of nonlinear systems of hyperbolic conservation laws in one spatial dimension. Our estimators rely on recent stability results by Bressan, Chiri and Shen and a novel method to compute negative order norms of residuals.  Numerical experiments show that the error estimator converges with the rate predicted by a-priori error estimates.  
\end{abstract}



\section{Introduction}
\label{sec:introduction}


We are concerned with the error analysis of numerical approximations of solutions to Cauchy-problems for one-dimensional systems of hyperbolic conservation laws defined by
\begin{align}
  \label{eq:pde}
  \begin{split}
    \vec{u}_t(t,x) + \vec{f}(\vec{u}(t,x))_x &= 0,\quad (t,x) \in [0,T] \times \R\\
    \vec{u}(0, x) = \vec{u}_0(x).
  \end{split}
\end{align}
Here $\vec{u} \,\colon\,  [0,T] \times \R \to \consDom \subset\R^m$ denotes the conserved quantity, $\vec{f} \in C^2(\consDom, \R^m)$ the flux function and $\vec{u}_0 \in L^{\infty}(\R, \consDom)$ is the initial condition. We assume that the system is strictly hyperbolic, i.e.~the Jacobian of the flux $\diff \vec{f}$ has $m$ distinct real  eigenvalues $\lambda_1(\vec{u}) < \cdots < \lambda_m(\vec{u})$ and the set of the corresponding eigenvectors forms an orthonormal basis. Furthermore, we assume that each characteristic field is either genuinely nonlinear or linear degenerate. Finally, we assume that the system is endowed with a strictly convex entropy $\entropy$ and corresponding entropy flux $\entropyFlux$, i.e.~a pair of functions $\entropy \,\colon\, \consDom \to \R$ and $\entropyFlux \,\colon\, \consDom \to \R$ satisfying $\diff \entropy\cdot \diff \vec{f} = \diff \entropyFlux$, see classical textbooks~\cite{dafermosHyperbolicConservationLaws2016} and~\cite{bressanHyperbolicSystemsConservation2000}.

Our goal is to provide rigorous a-posteriori error estimates for numerical schemes and for technical simplicity, to outline the ideas, we restrict ourselves to first order finite-volume schemes. Extending our method to higher order schemes is subject of future research. In general, a-posteriori error estimates require uniqueness of solutions and rely on a suitable stability theory of the underlying PDEs.
Indeed, it is well known that for one dimensional hyperbolic systems as specified above and data that have small total variation a Lipschitz continuous semi-group of entropy solutions exist. To be more precise, we rely on stability results provided by Bressan, Chiri and Shen~\cite{bressanPosterioriErrorEstimates2021a}. In~\cite{bressanPosterioriErrorEstimates2021a} these results are used to prove convergence of a quite general class of numerical schemes provided numerical solutions satisfy an a-posteriori verifiable criterion, i.e.~bounded total variation and bounded oscillation.  In contrast, we aim at providing a fully a-posteriori error estimate since this leads to sharper bounds and will serve as a basis for adaptive numerical schemes.

In this work, we consider solutions to the Cauchy-problem~(\ref{eq:pde}) that are constant outside of some space-time cylinder $\stDom := [0,T]\times\dom$ where $ \dom\subset\R$ is some non-empty interval. In order to define the finite-volume scheme, let $\stDom$ be discretized by grid points $(t^n, x_{j-\frac12})$  where  $n\in\{0,\ldots,N\}$ and $j\in\{0,\ldots, J\}$.  We denote  the corresponding time-steps by $\Delta t^{n+\frac12} := t^{n+1} - t^n$ and cell widths by $\Delta x_{j} := x_{j+\frac12}  - x_{j-\frac12} $. Finite-volume schemes approximate the solution in each cell by spatial averages
\[
  \num{\vec{u}}^n_{j} \approx \frac{1}{\Delta x_j}\int_{x_{j-\frac12}}^{x_{j+\frac12}} \vec{u}(t^n, x) \,dx,
\]
 where $\num{\vec{u}}^0$ consists of piece-wise averages of the initial condition $\vec{u}_0$ and the discrete evolution of these spatial averages is given by
\begin{equation}
  \label{eq:finite-volumes}
  \num{\vec{u}}^{n+1}_{j} = \num{\vec{u}}^n_{j} - \frac{\Delta t^{n+\frac12}}{\Delta x_j}\left( \numflux\left(\num{\vec{u}}^n_{j}, \num{\vec{u}}^n_{j+1}\right) - \numflux\left(\num{\vec{u}}^n_{j-1}, \num{\vec{u}}^n_{j}\right)  \right).
\end{equation}
The numerical flux $\numflux  \,\colon\, \consDom\times\consDom \to \R^m$  approximates $\vec{f}$ at the cell boundary $x_{j+\frac12}$. Since we consider explicit time-integration, the CFL-condition
restricting the timestep is necessary for the stability of the scheme.
For details on finite-volume schemes we refer to established literature, e.g.~\cite{godlewskiHyperbolicSystemsConservation1991},~\cite{toroRiemannSolversNumerical2009}. 

Finite volume solutions are interpreted as  functions in $L^{\infty}(\stDom)$ by setting
\[
  \num{\vec{u}}(t,x) = \vec{u}_j^n \,\text{ if }\, (t,x) \in  K_j^n:=[t^{n}, t^{n+1}) \times (x_{j-\frac12}, x_{j+\frac12}),
\]
i.e.~piecewise constant on rectangular space-time cells.   We define the space of such piecewise constant functions by
\begin{multline}
  \label{eq:kjn-pw-constant-def}
  \pwConst(\stDom) :=
  \Bigg\{  
  v \,\colon\, [0,T] \times \R \to \R \,\colon\,  v\Big|_{K_j^n}\equiv const\,\,\forall n\in\{0,\ldots, N\},\,j\in\{0,\ldots,J\},\\
 v\Big|_{[0,T] \times\left( -\infty,\, x_{-\frac12}\right)}  \equiv const,\,\,   v\Big|_{[0,T] \times \left(x_{J+\frac12},\, \infty\right)} \equiv const 
  \Bigg\}.
\end{multline}

In~\cite{bressanPosterioriErrorEstimates2021a}, estimates of the $L^\infty([0, T), L^1(\R))$-error for solutions $\num{\vec{u}}\in\pwConst{(\stDom)}^m $ produced by first order finite-volume schemes for Cauchy-problems are presented.
These estimates  hinge on suitable consistency and stability  results. 
The error estimates presented in~\cite{bressanPosterioriErrorEstimates2021a} depend on a rather abstract consistency criterion, equations~\eqref{eq:P_eps} and \eqref{eq:P_eps_entropy} below. It is proven in~\cite{bressanPosterioriErrorEstimates2021a}  that certain numerical schemes produce solutions satisfying this criterion with $\varepsilon$ going to zero under mesh refinement.

Indeed, error estimates are provided in~\cite{bressanPosterioriErrorEstimates2021a} if, firstly, there exists an $\varepsilon > 0$ such that the weak residual of the numerical solution can be bounded, i.e.
\begin{align}
  \label{eq:P_eps}
  \begin{split}
    &\left|   \int_{\R} \num{\vec{u}}(t^1, x)\varphi(t^1, x)\, dx - \int_{\R} \num{\vec{u}}(t^2, x)\varphi(t^2, x)\, dx + \int_{t^1}^{t^2} \int_{\R} \left(  \num{\vec{u}}\varphi_t + \vec{f}(\num{\vec{u}})\varphi_x \right)\,dx\,dt\right|\\
    &\quad\leq C\varepsilon \|\varphi\|_{W^{1,\infty}} (t^2 - t^1)\sup_{t\in[t^1, t^2]} TV[\num{\vec{u}}(t, \,\cdot\,)] \quad \forall  \varphi \in C^{1}_c(\R^2), \quad \forall\, 0 \leq t^1 < t^2 \leq T, 
  \end{split}      
\end{align}
where $C > 0$ is a constant, $\varphi$ denotes a \emph{scalar} test function with compact support, and $t^1 < t^2$ with $t^1, t^2 \in [0, T]$  (cf.~condition $(\mathrm{P}_{\varepsilon})$ in~\cite{bressanPosterioriErrorEstimates2021a}). In addition, the same $\varepsilon > 0$ needs to bound the violation of the entropy inequality, i.e.
  \begin{align}
    \label{eq:P_eps_entropy}
    \begin{split}
      & \int_{\R} \entropy\left(\num{\vec{u}}(t^2, x)\right)\varphi(t^2, x)\, dx - \int_{\R} \entropy\left(\num{\vec{u}}(t^1, x)\right)\varphi(t^1, x)\, dx \\
      &\quad\quad\quad\quad\quad- \int_{t^1}^{t^2} \int_{\R} \left(  \entropy\left(\num{\vec{u}}\right)\varphi_t + \entropyFlux(\num{\vec{u}})\varphi_x \right)\,dx\,dt\\
      &\leq C\varepsilon \|\varphi\|_{W^{1,\infty}} (t^2 - t^1)\sup_{t\in[t^1, t^2]} TV[\num{\vec{u}}(t, \,\cdot\,)] \quad \forall  \varphi \in C^{1}_c(\R^2),\, \varphi \geq 0,\,\, \forall\, 0 \leq t^1 < t^2 \leq T.
    \end{split}
  \end{align}

  These consistency results are combined in~\cite{bressanPosterioriErrorEstimates2021a} with stability estimates that are obtained  by comparing $\num{\vec{u}}$ with the solution of a linearized PDE  in regions where $\num{\vec{u}}$  is smooth, whereas in regions of large jumps $\num{\vec{u}}$ is compared to solutions of Riemann problems. The key condition a numerical solution needs to satisfy so that the results from~\cite{bressanPosterioriErrorEstimates2021a} are applicable is uniformly  bounded total variation and uniformly bounded oscillations. Both these properties are a-posteriori verifiable. The conceptual difference between the work at hand and~\cite{bressanPosterioriErrorEstimates2021a} is that we aim at actually computing values for $\varepsilon$ from the numerical solution instead of proving a-priori convergence rates. This leads to smaller values of $\varepsilon$.

    The starting point of our work is the observation that the estimates (\ref{eq:P_eps}) and~(\ref{eq:P_eps_entropy}) measure residuals in the $W^{-1,1}$-norm. Here, ``residual''  refers to the quantity that is obtained when the numerical solution is inserted into the weak form of the PDE and of the entropy inequality, respectively.
    Let us note that by the very definition of residual, i.e. the left hand sides of \eqref{eq:P_eps} and \eqref{eq:P_eps_entropy}, this approach can only provide  error estimates for numerical solutions such that $\num{\vec u}(t,x) \in \consDom$ for almost all $(t,x)$.

    The consistency estimates in~\cite{bressanPosterioriErrorEstimates2021a} are already local in time but in order to obtain results that can serve as a basis for mesh-adaptation we also need to localize them in space. Moreover, using an infinite-dimensional space of test functions is not practical and we need to show that this can be replaced by a finite-dimensional  space (that is in fact three-dimensional).

  Let us compare our results with a-posteriori error estimates for hyperbolic conservation laws that are available in the literature.
  Early results  concerning wave-front tracking and the Glimm-scheme are due to~\cite{laforestmarcetiennePosterioriErrorEstimate2001, kimAdaptiveVersionGlimm2010, laforestPosterioriErrorEstimate2004}. For scalar problems, strong stability results are available: $L^1$-contraction based on Kru\v{z}hkov doubling of variables~\cite{kruzkovFirstOrderQuasilinear1970}  was used for a-posteriori estimates in~\cite{kronerPosterioriErrorEstimates2000}. Results for scalar problems in 1D provided  in~\cite{karni2005local}, based on stability results in~\cite{nessyahu1994convergence}, are conceptually somewhat similar to our results since they measure residuals in dual norms. However, we are able to provide rigorous upper bounds for these dual norms by introducing a suitable  projection.

  Error estimates for systems of hyperbolic conservation laws based on relative entropy were obtained in~\cite{jovanovicErrorEstimatesFinite2006, giesselmannPosterioriAnalysisDiscontinuous2015a,dednerPosterioriAnalysisFully2016}. They do not provide informative bounds if the exact solution is discontinuous, since the classical relative entropy does not provide stability for discontinuous solutions.
  In one spatial dimension, the limitations of the relative entropy method can be mitigated by the theories of shifts and $a$-contractions, see~\cite{krupa2019uniqueness} and~\cite{chen2022uniqueness}. 
  This opens the door for improved relative entropy based error estimators, see the forthcoming paper~\cite{giesselmann2022theory}.
  
  Finally, goal-oriented error estimates were considered by~\cite{hartmannAdaptiveDiscontinuousGalerkin2002} but require the solution of adjoint problems for which well-posedness in case of systems is unclear, even in one space dimension, if the primal solution is discontinuous.

  The outline of this paper is as follows. In the second section we explain how the bounds for the residual can be localized and be made computable. In Section~3 we discuss how oscillation bounds that enter the error estimates can be computed efficiently. We also explain how the bounds for residuals and oscillations can be combined to provide  a-posteriori error estimates.  In the final section we present numerical experiments showing that the error bounds scale as predicted by the a-priori analysis in~\cite{bressanPosterioriErrorEstimates2021a}.


\section{Weak residuals}\label{sec:weak-formulation}

As pointed out before, (\ref{eq:P_eps}) and (\ref{eq:P_eps_entropy}) measure the dual norm of the residual.  In this chapter we replace global in space dual norms by local in space dual norms and we show how to compute them.

\subsection{Localizing weak residuals}\label{sec:localizing}  
We define a weak residual on each space-time cell and the sum of  their $W^{-1,1}$-norms will be a key ingredient in the error estimator in Section~\ref{sec:analys-postr-estim}.

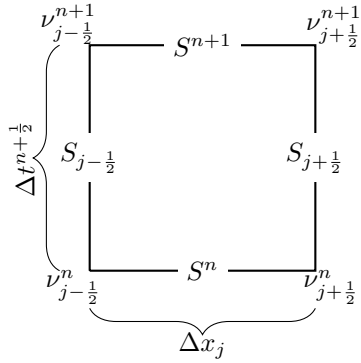
\begin{figure}[!htb]
  \centering


\begin{tikzpicture}
  \path (0,0) coordinate (A)
  (3,0) coordinate (B)
  (3,3) coordinate (C)
  (0,3) coordinate (D);

  \draw[thick] (A) -- (B) -- (C) -- (D) -- (A);

  \draw (0,1.5) + (0:0mm) node[fill=white]  {$S_{j-\frac12}$};
  \draw (3,1.5) + (0:0mm) node[fill=white]  {$S_{j+\frac12}$};
  \draw (1.5,0) + (-90:0mm) node[fill=white]  {$S^{n}$};
  \draw (1.5,3) + (90:0mm) node[fill=white]  {$S^{n+1}$};
  
  \draw (A) + (225:3mm) node  {$\nu_{j-\frac12}^{n}$};
  \draw (B) + (-45:3mm) node  {$\nu_{j+\frac12}^{n}$};
  \draw (C) + (45:3.5mm) node  {$\nu_{j+\frac12}^{n+1}$};
  \draw (D) + (135:4mm) node  {$\nu_{j-\frac12}^{n+1}$};

  \draw [decorate,decoration={brace,amplitude=10pt, mirror},yshift=-5mm]
  (0.0,0.0) -- (3,0) node [black,midway, yshift=-5mm] 
  {$\Delta x_{j}$};
  \draw [decorate,decoration={brace,amplitude=10pt},xshift=-4mm]
  (0.0,0.0) -- (0,3) node [black,midway, xshift=-5mm,rotate=90] 
  {$\Delta t^{n+\frac12}$};
\end{tikzpicture}

  \caption{Space-time element $K_j^n$ with corresponding edges $S$. }\label{fig:space-time-elm}
\end{figure}

As a first step, we provide some notation for space-time cells. Let the set of nodes of a cell $K_j^{n}$ be denoted by $\nodeSet\left(K_j^{n}\right)$, i.e. 
\[
  \nodeSet\left(K_j^{n}\right):=\left\{\node^{n}_{j-\frac12},\, \node^{n}_{j+\frac12},\, \node^{n+1}_{j+\frac12},\, \node^{n+1}_{j-\frac12}\right\}, \text{ with } \node^\alpha_\beta:=(t^\alpha,\, x_\beta)^T
\]
and the set of edges of $K_j^n$ by $\edges{K_j^n}$, as illustrated in Figure~\ref{fig:space-time-elm}, where
\begin{align}
  \label{eq:surfaces-def}
  \begin{split}
    &\edges{K_j^n} := \{S_{j-\frac12},\, S_{j+\frac12},\, S^{n},\, S^{n+1}\},\\
    &S_{j-\frac12} := \overline{\nu_{j-\frac12}^{n}\nu_{j-\frac12}^{n+1}},\quad
    S_{j+\frac12} := \overline{\nu_{j+\frac12}^{n}\nu_{j+\frac12}^{n+1}},\\
    &S^{n} := \overline{\nu_{j-\frac12}^{n}\nu_{j+\frac12}^{n}},\quad
    S^{n+1} := \overline{\nu_{j-\frac12}^{n+1}\nu_{j+\frac12}^{n+1}}.
  \end{split}
\end{align}

Given a  piecewise-constant function, $\num{\vec{u}} \in\pwConst{(\stDom)}^m$, that is constant for $x \notin \Omega$, we aim to compute the smallest $\varepsilon$ satisfying both inequalities (\ref{eq:P_eps}) and (\ref{eq:P_eps_entropy}). Neither the  weak residual,~(\ref{eq:P_eps}), nor the weak entropy dissipation residual,~(\ref{eq:P_eps_entropy}), should be split into single cells directly, i.e.~by splitting the domains of the integrals. Such a naive localization would be inconsistent since integration by parts  in space leads to boundary terms.   In the following we first treat the weak residual (\ref{eq:P_eps}) and apply the same procedure to the entropy dissipation residual (\ref{eq:P_eps_entropy}).
  
Localizing the weak residual  on a space-time cell $K_j^n$ requires fluxes at the edges $S_{j-\frac12}$ and $S_{j+\frac12}$. A natural choice is to assign numerical fluxes $\num{\vec{f}}$ of a first-order finite-volume scheme.   Note that the numerical flux function does not have to be the one that was used to produce $\num{\vec{u}}$, but we will actually use this flux for efficiency reasons, see Remark \ref{rem:stencil}

\begin{definition}\label{def:1}
 Let the space-time domain be partitioned as\, $\stDom = \bigcup_{j,n}\overline{K^n_j}$ and some numerical flux function $\num{\vec{f}}$ as well as $\num{\vec{u}} \in \pwConst(\Omega_T)^m$ be given. Then, the linear local weak residual operator on $K_j^n$  is defined as
    \begin{align}
      \begin{split}\label{eq:b-op-def}
      &\opB_j^n[\num{\vec{u}}] \,\colon\,  W^{1,\infty}(K_j^n)  \to \R^m,\\
        &\varphi \mapsto \int_{x_{j-\frac12}}^{x_{j+\frac12}} \num{\vec{u}}^{n}_{j} \tfct(t^{n},x)\, dx
        - \int_{x_{j-\frac12}}^{x_{j+\frac12}} \num{\vec{u}}^{n+1}_{j}\tfct(t^{n+1}, x)\, dx\\
        &\quad\quad\quad\quad + \int_{K_j^n} \num{\vec{u}}^{n}_{j}\tfct_t(t,x) + \vec{f}(\num{\vec{u}}^{n}_{j})\tfct_x(t,x)\,dx\,dt\\
        &\quad\quad\quad\quad +\int_{t^{n}}^{t^{n+1}}\!\!\! \numflux\left(\num{\vec{u}}^n_{j-1}, \num{\vec{u}}^n_{j}\right) \tfct(t, x_{j-\frac12}) \,dt -\int_{t^{n}}^{t^{n+1}} \!\!\!\numflux\left(\num{\vec{u}}^n_{j}, \num{\vec{u}}^n_{j+1}\right) \tfct(t,x_{j+\frac12}) \,dt.
      \end{split}
    \end{align}
 For the reason of consistency, we denote the constant value on the left- and right-hand side boundary, i.e.~in $[0,T]\times(-\infty, \, x_{-\frac12})$ and $[0,T]\times(x_{J+\frac12}, \infty)$, by  $\num{\vec{u}}_{-1}$ and $\num{\vec{u}}_{J+1}$, respectively. 
\end{definition}
\begin{remark}\label{rem:stencil}
  The localized  weak residual has the same locality as the finite volume scheme, i.e.~the stencil consists of only three cells. Therefore, well-established implementation techniques, such as distributed memory parallelization with spatial domain decomposition, can be applied for computing $\opB_j^n$ as well. In addition, the numerical fluxes of the scheme may be saved and recycled for evaluation of $\opB_j^n$ increasing the efficiency of the computations drastically.  
\end{remark}
We denote the averages $\fint_E g\, dS$ of a sufficiently regular function $g \,\colon\, K_j^n \to \R$ restricted onto edges $S_{j\pm \frac12}$ and $S^{n,n+1}$ in $\edges{K^n_j}$  by
\begin{align}
  \label{eq:def-average}
  \begin{split}
    &\av{g}_{\zeta} := \fint_{t^n}^{t^{n+1}}\!\!\!\!\lim_{x\to x_\zeta} g(t,x)\,dt = \frac{1}{\Delta t^{n+1/2}}\int_{t^n}^{t^{n+1}}\!\!\!\!\lim_{x\to x_\zeta} g(t,x)\,dt,\\
    &\av{g}^{\mu }:= \fint_{x_{j-\frac12}}^{x_{j+\frac12}}\lim_{t\to t^\mu}g(t,x)\,dx = \frac{1}{\Delta x_{j}}\int_{x_{j-\frac12}}^{x_{j+\frac12}}\lim_{t\to t^\mu}g(t,x)\,dx,
    \end{split}
\end{align}
where $\zeta\in\{j-\frac12, j+\frac12\}$ and $\mu\in\{n,n+1\}$, respectively.  Since the numerical solution is constant  in the interior of the space-time cell $K_j^n$, applying the Gauss-theorem on the space-time integral in~\eqref{eq:b-op-def} allows to rewrite $\opB_j^n$ as follows:
\begin{align}
  \label{eq:b-op-simplified}
  \begin{split}
    \opB_j^n[\num{\vec{u}}](\tfct) &= \Delta x_{j}\left(\num{\vec{u}}^n_{j} - \num{\vec{u}}^{n+1}_{j}\right)\avPhi^{n+1}
    + \Delta t^{n+\frac12}\vec{f}(\num{\vec{u}}_{j}^n)\left( \avPhi_{j+\frac12} - \avPhi_{j-\frac12}  \right)\\
    &+\Delta t^{n+\frac12}\left(\numflux\left(\num{\vec{u}}^n_{j-1}, \num{\vec{u}}^n_{j}\right) \avPhi_{j-\frac12} - \numflux\left(\num{\vec{u}}^n_{j}, \num{\vec{u}}^n_{j+1}\right) \avPhi_{j+\frac12} \right)
  \end{split}
\end{align}
where the term $\Delta x_{j}\num{\vec{u}}^n_{j} \avPhi^{n}$ cancels out.
From now on, we use  \eqref{eq:b-op-simplified} as the definition of $\opB_j^n$. The values of $\opB_j^n[\num{\vec{u}}]( \tfct)$ are determined by the values $\num{\vec{u}}_j^{n+1}$ and  $\num{\vec{u}}_j^n$, the numerical fluxes to the neighboring space-time cells and by averages of the test function $\tfct$ along three edges of $K_j^n$.

\subsection{Computable bounds for local weak residuals}
Computing bounds in (\ref{eq:P_eps}) and (\ref{eq:P_eps_entropy}) seems to require testing with an infinite set of test functions which is not practical.  We devise a suitable projection operator mapping Lipschitz continuous test functions into a finite dimensional space that allows us to bound the $W^{-1,1}$-norms of weak residuals by their operator norms on this finite dimensional space. We give the details of this procedure for (\ref{eq:P_eps});  (\ref{eq:P_eps_entropy}) can be treated analogously.

Equation~\eqref{eq:b-op-simplified} shows that the value of $ \opB^n_j[ \num{\vec{u}}] (\varphi)$ depends on a small number of scalar values and motivates the search for a low dimensional test space whose elements are characterized by averages on the edges in $\edgess{K_j^n} := \left\{S_{j-\frac12}, \, S^{n+1}, \, S_{j+\frac12}  \right\}$, i.e.~the edges of $K_j^n$ that enter~(\ref{eq:b-op-simplified}).
Such a space is provided by affine linear functions on $K_j^n$, i.e.,
  \begin{multline}
    \label{eq:hat-fcts-subspace-t}
    \basisSub(K^{n}_j) := \biggl\{v\,\colon\, K^{n}_j\to \R,\, v=\alpha_1 + \frac{(t^{n+1} - t )}{\Delta t^{n+\frac12}}\alpha_2 + \frac{( x_{j+\frac12} - x)}{\Delta x_{j}}\alpha_3,\,\\ \alpha_i \in \R,\, i\in\{1,2,3\}  \biggr\},
  \end{multline}
  which we equip with the $W^{1,\infty} $ norm. 
  We define a projection from the test functions to $\basisSub(K^{n}_j)$ that preserves the averages on $\edgess{K_j^n}$:   
  \begin{align}
    \label{eq:p-def}
    \begin{split}
      &\opTestToFem^{n}_j \,\colon\, W^{1,\infty}(K^{n}_j) \to \basisSub(K^{n}_j),\\
      &\tfct \mapsto  \opTestToFem^{n}_j\tfct  \text{ with } \fint_S{\opTestToFem^{n}_j\tfct}\, dS = \fint_S{\tfct}\, dS,\,  S \in \edgess{K_j^n},
    \end{split}
  \end{align}
 This allows us to efficiently compute an estimate for $\| \opB_j^n[\num{\vec{u}}]\|_{\mathcal{L}(W^{1,\infty}(K^n_j),\, \R^m)}$, given some piecewise-constant function $\num{\vec{u}}\in \pwConst(\stDom)^m$.
  \begin{proposition}
    \label{prop:Pjn-estimate}
   The operator  $\opTestToFem^{n}_j$  is well-defined and linear and for any $\tfct \!\in\! W^{1,\infty}(K^{n}_j)$
    \begin{equation}
      \label{eq:P-is-cont}
      \| \opTestToFem^{n}_j\tfct \|_{W^{1,\infty}(K^{n}_j)} \leq  \max\left\{3,\, \sqrt{8+8(c_j^n)^2}\right\}\|\tfct \|_{W^{1,\infty}(K^{n}_j)},
    \end{equation}
    where  $c_j^n := \frac{\Delta t^{n+\frac12}}{\Delta x_j}$.
  \end{proposition}
   \begin{remark}
   Note that in practice one would want to use (\ref{eq:finite-volumes}) with a CFL-conditions which means that $c_j^n > 0$ is bounded from above by the maximal wave speed.
 \end{remark}
 
  \begin{proof}    
 Since  the cell $K_j^n$ is a domain with Lipschitz-boundary, the trace of any $\tfct  \in W^{1,\infty}(K_j^n)$
    is well defined  on each edge.

 Given the  averages of the test function $\vec{\avPhi} = \left(\avPhi_{j-\frac12},\, \avPhi^{n+1},\, \avPhi_{j+\frac12}  \right)^T$,  the coefficients $\vec{\alpha} := (\alpha_1, \alpha_2,\alpha_3)^T$ of $\opTestToFem^{n}_j\tfct\in\basisSub(K^n_j)$  are determined by the solution of the linear system    
    \begin{equation}
      \label{eq:system-P-coeff}
      \begin{pmatrix}
        1&1/2&-1/2\\
        1&0&0\\
        1&1/2&1/2
      \end{pmatrix}
      \vec{\alpha}
      =
      \vec{\avPhi}
      \quad \Leftrightarrow \quad
      \vec{\alpha} =
      \begin{pmatrix}
        0&1&0\\
        1&-2&1\\
        -1&0&1
      \end{pmatrix}
     \vec{\avPhi}.
   \end{equation}
    The matrix  describing the mapping from the averages to the coefficients $\vec{\alpha}$, as defined in (\ref{eq:system-P-coeff}), is invertible. Hence, prescribing the averages along the edges yields a well-defined operator $\opTestToFem_j^n$. 

The $W^{1,\infty}$-norm of $\opTestToFem_j^n\tfct$ can be estimated as follows. 
    \begin{align}
      \label{eq:norm-Pphi-estimate}
      \begin{split}
        \|\opTestToFem_j^n\tfct\|_{W^{1,\infty}(K_j^n)} =& \max\left\{\| \opTestToFem_j^n\tfct \|_{L^\infty(K_j^n)}, \| \diff(\opTestToFem_j^n\tfct) \|_{L^\infty(K_j^n)}\right\}\\
        \stackrel{~\eqref{eq:hat-fcts-subspace-t}}{=}&\left\{ 
                                                  \max\{ (\opTestToFem_j^n\tfct)(\nu) \,\colon\, \nu \in \nodeSet(K_j^n) \}
                                                  ,\,
                                                  \max\left\{
                                                  \frac{|\alpha_2|}{\Delta t^{n+\frac12}},\, \frac{|\alpha_3|}{\Delta x_j}
                                                  \right\}
                                                  \right\}.
      \end{split}
    \end{align}
    We first estimate $ \| \opTestToFem_j^n\tfct \|_{L^\infty(K_j^n)}$:
    \begin{multline}
        \max\{ (\opTestToFem_j^n\tfct)(\nu) \,\colon\, \nu \in \nodeSet(K_j^n) \} \\= \max\left\{|\alpha_1|,\,
                                                                              |\alpha_1 + \alpha_2|,
                                                                              |\alpha_1 + \alpha_3|,
                                                                              |\alpha_1 + \alpha_2 + \alpha_3| \right\}
         \stackrel{~(\ref{eq:system-P-coeff})}{\leq} 3\|\tfct \|_{W^{1,\infty}}.
          \label{eq:max_pf_estimate_val}
    \end{multline}
   Next, we estimate $ \| \diff(\opTestToFem_j^n\tfct) \|_{L^\infty(K_j^n)}$. The partial derivative in $x$-direction 
  satisfies
   %
    \begin{align*}
      |\alpha_3| = \left| \av{\vec{\tfct}}_3 - \av{\vec{\tfct}}_1  \right| &= \frac{1}{\Delta t^{n+\frac12}} \left| \int_{t^n}^{t^{n+1}} \tfct(t, x_{j+\frac12}) - \tfct(t, x_{j-\frac12})\,dt
      \right|
       \leq \Delta x_j\|\tfct\|_{W^{1,\infty}}.
    \end{align*}
  The partial derivative in  $t$-direction is estimated as follows.
  There exist $\xi_1 \in S_{j-\frac12}$, $\xi_2 \in S^{n+1}$ and $\xi_3 \in S_{j+\frac12}$, depicted in Figure~\ref{fig:space-time-elm}, such that
    \begin{align}
      |\alpha_2| =  \left| \av{\vec{\tfct}}_1 - 2\av{\vec{\tfct}}_2 + \av{\vec{\tfct}}_3  \right|  = \left| \tfct(\xi_1) - 2\tfct(\xi_2) + \tfct(\xi_3) \right|.
    \end{align}
    Since, $\|\tfct\|_{W^{1.\infty}}$ is the lipschitz constant of $\tfct$ 
    \begin{multline}
        \left| \tfct(\xi_1) - \tfct(\xi_2) - (\tfct(\xi_3) - \tfct(\xi_2)) \right| 
      \leq   \|\tfct\|_{W^{1,\infty}} ( \left| \xi_1 - \xi_2  \right| +\left| \xi_3 - \xi_2  \right|)\\
        \leq 2\sqrt{2}\|\tfct\|_{W^{1,\infty}}\left| \left(\Delta t^{n+\frac12}, \, \Delta x_j\right)^T \right|.
      \end{multline}
    Thus, using the relation $\Delta t^{n+\frac12} = c_j^n\Delta x_j$, we have $|\alpha_2| \leq \sqrt{8+8(c_j^n)^2}\Delta t^{n+\frac12}\|\tfct\|_{W^{1,\infty}}$.

    Summarizing the above results, we obtain
    \[
      \|\opTestToFem_j^n\tfct\|_{W^{1,\infty}(K_j^n)} \leq \max\left\{3,\, \sqrt{8+8(c_j^n)^2}\right\}\|\tfct\|_{W^{1,\infty}(K_j^n)}
    \]
    and remark that the estimate \emph{does not depend on the grid size if the CFL ratio $\frac{\Delta t^{n+\frac12}}{\Delta x_j}$ is uniformly bounded  in each timestep}.
  \end{proof}
 Let us recall that $\opTestToFem$ preserves averages, i.e.
    \begin{equation}
      {\left(\av{\opTestToFem_j^{n}\tfct}\right)}_{j\pm\frac12} = \avPhi_{j\pm\frac12} \text{ and }{\left(\av{\opTestToFem_j^{n,\pm}\tfct}\right)}^{n+1} = \avPhi^{n+1}
      \label{eq:proof-bjn-under-p}    
    \end{equation}
    and, thus, $\opB^{n}_j$ is invariant when $\opTestToFem^{n}_j$ is applied to the test function.
  \begin{proposition}\label{prop:B-invariant-under-P}
    Let $\num{\vec{u}} \in {\pwConst(\stDom)}^m$ be a piecewise-constant function   and $\tfct \in W^{1,\infty}(K^n_j)$ on a space-time cell $K^n_j \subset \stDom$. Then
    \begin{equation}
      \label{eq:bjn-invariant-under-P}
      \opB_j^n[\num{\vec{u}}]( \tfct) = \opB_j^n[\num{\vec{u}}] (\opTestToFem_j^n\tfct).
    \end{equation}
  \end{proposition}
  Proposition \ref{prop:B-invariant-under-P} implies that for any $\num{\vec{u}} \in {\pwConst(\stDom)}^m$ the operator norm of $\opB_j^n[\num{\vec{u}}]$  on the set of all test functions   can be estimated by the operator norm of $\left. \opB_j^n[\num{\vec{u}}]\right|_\basisSub$  in the following manner.
  \begin{proposition}\label{prop:bjn-estimate}
    For any  piecewise-constant function $\num{\vec{u}}\in {\pwConst(\stDom)}^m$ 
    \begin{equation}
      \label{eq:estimate-B}
      \|\opB_j^n[\num{\vec{u}} ]\|_{\mathcal{L}(W^{1,\infty}(K^j_n), \R^m)}
      \leq
      \max\left\{3,\, \sqrt{8+8(c_j^n)^2}\right\}\|\opB_j^n[\num{\vec{u}}]\|_{\mathcal{L}(\basisSub(K^j_n), \R^m)},
    \end{equation}
    with $c_j^n$ as in Proposition \ref{prop:Pjn-estimate}.
  \end{proposition}
  \begin{proof}
    Using the definition of the norm of a linear operator and applying previous results yields:
    \begin{align}
      \begin{split}
      \label{eq:estimate-B-proof-1}
      \|\opB_j^n[\num{\vec{u}}]\|_{\mathcal{L}(W^{1,\infty}(K^j_n), \R^m)}
      &= \sup_{\|\tfct\|_{W^{1,\infty}\left(K^{n}_j\right)} \neq 0} \frac{\left| \opB_j^n[\num{\vec{u}}]( \tfct) \right|}{\| \tfct \|_{W^{1,\infty}\left(K^{n}_j\right)}}\\
      &\stackrel{\eqref{eq:bjn-invariant-under-P}}{=}
        \sup_{\|\tfct\|_{W^{1,\infty}\left(K^{n}_j\right)} \neq 0} \frac{\left| \opB_j^n[\num{\vec{u}}] (\opTestToFem_j^n\tfct) \right|}{\| \tfct \|_{W^{1,\infty}\left(K^{n}_j\right)}}\\
      &\stackrel{\eqref{eq:P-is-cont}}{\leq}
        \max\left\{3,\, \sqrt{8+8(c_j^n)^2}\right\}\sup_{\|\tfct\|_{W^{1,\infty}\left(K^{n}_j\right)} \neq 0}
        \frac{\left| \opB_j^n[\num{\vec{u}}] (\opTestToFem_j^n\tfct)  \right|}
        {\| \opTestToFem_j^n\tfct \|_{\basisSub\left(K^{n}_j\right)}}\\
      &=
        \max\left\{3,\, \sqrt{8+8(c_j^n)^2}\right\}\sup_{\ttfct \in \basisSub(K^{n}_j),\, \|\ttfct\|_{W^{1,\infty}\left(K^{n}_j\right)} = 1} \left| \opB_j^n[\num{\vec{u}}]( \ttfct)  \right|\\
        &=    \max\left\{3,\, \sqrt{8+8(c_j^n)^2}\right\}\|\opB_j^n[\num{\vec{u}} ]\|_{\mathcal{L}(\basisSub(K^j_n), \R^m)}.
      \end{split}
    \end{align}
    \
  \end{proof}

  In the next step, we  show that the weak residual on $[t^n, t^{n+1}]\times\dom$, can be estimated by the sum of the local residua. To this end, we define
  \begin{equation}
    \opB^n[\num{\vec{u}}]\,\colon\, W^{1,\infty}([t^n, t^{n+1}]\times\dom) \to \R^m,\quad
    \opB^n[\num{\vec{u}}] (\tfct) \!=\! \sum_{j=0}^{J} \opB_j^n[\num{\vec{u}}] \left(\left.\tfct\right|_{K^n_j}\right) 
        \label{eq:def-global-b},
  \end{equation}
  where the numerical fluxes on the boundary of $\dom$ are obtained by means of some constant (in time and space) outer states $\num{\vec{u}}_{-1}$ and $\num{\vec{u}}_{J + 1}$.  The local spaces of  affine linear functions are gathered to define the space of piecewise linear functions $\basisSub([t^n, t^{n+1}]\times\dom) :=  \bigoplus_{j} \basisSub(K^{n}_j)$ with the corresponding norm: $v\in\basisSub([t^n, t^{n+1}]\times\dom),\, \|v\|_{\basisSub}:=\max_{j}\left\{ \left\|v\big|_{K_j^n}\right\|_{W^{1,\infty}(K_j^n)} \right\}$.

  Recall that the definition of $\opB_j^n$ in (\ref{eq:b-op-simplified}) corresponds to the weak residual, i.e.~the left-hand side of the inequality~(\ref{eq:P_eps}), on $K_j^n$ except that we added numerical fluxes. It is a tedious but straight-forward calculation to show that the numerical fluxes cancel upon summation and  $\opB^n$ coincides with the weak residual.

     As a consequence, the operator norm of $\opB^n[\num{\vec{u}}]$ can be bounded by summing up the local contributions of the operator norms of $\opB_j^n[\num{\vec{u}}]$.
  \begin{proposition}\label{prop:global-bjn}
For any  piecewise-constant function $\num{\vec{u}}\in {\pwConst(\stDom)}^m$ 
    \begin{equation}
      \label{eq:estimate-B-global}
      \|\opB^n[\num{\vec{u}} ]\|_{\mathcal{L}(W^{1,\infty}([t^n, t^{n+1}]\times\dom), \R^m)}  \leq C^n \sum_{j=0}^{J} \left\|  \opB_j^n[\num{\vec{u}} ]\right\|_{\mathcal{L}(\basisSub(K_j^n) \R^m)}
    \end{equation}
    holds, where $C^n:=\max\left\{3,\, \max_{j=1,\ldots, J}\sqrt{8+8(c_j^n)^2}\right\}$.
  \end{proposition}
  \begin{proof}
    \begin{align*}
      \begin{split}
      \|\opB^n[\num{\vec{u}}] \|_{\mathcal{L}(W^{1,\infty}([t^n, t^{n+1}]\times\dom), \R^m)}
      &= \sup_{\|\tfct\|_{W^{1,\infty}\left([t^n, t^{n+1}]\times\dom\right)} =1} \left| \opB^n[\num{\vec{u}}](\tfct) \right|\\
      &= \sup_{\|\tfct\|_{W^{1,\infty}\left([t^n, t^{n+1}]\times\dom\right)} =1} \left|\sum_{j=0}^{J} \opB_j^n[\num{\vec{u}}]\left( \left.\tfct\right|_{K_j^n}\right)  \right|\\
      &\stackrel{\Delta\text{-ineq.},~\eqref{eq:bjn-invariant-under-P}}{\leq}
        \!\!\!\!\!\!\!\! \sup_{\|\tfct\|_{W^{1,\infty}\left([t^n, t^{n+1}]\times\dom\right)} =1}\sum_{j=0}^{J} \left|\opB_j^n[\num{\vec{u}}] \left(\left.\opTestToFem_j^n\tfct\right|_{K_j^n}\right)  \right|\\
      & \stackrel{~(\ref{eq:estimate-B})}{\leq}  C\sum_{j=0}^{J} \sup_{\|\ttfct\|_{\basisSub\left(K_j^n\right)} =1} \left|  \opB_j^n[\num{\vec{u}}] (\ttfct) \right|.
      \end{split}
    \end{align*}
    \
  \end{proof}

Since the space $\basisSub([t^n, t^{n+1}]\times\dom)$ is finite-dimensional, the bound (\ref{eq:estimate-B-global}),  together with the total variation of $\num{\vec{u}}$, allows to compute  $\varepsilon$ satisfying the weak residual inequality (\ref{eq:P_eps}). In the following we proceed with localizing the weak entropy dissipation residual (\ref{eq:P_eps_entropy}) similarly. To this end, we first rewrite the left-hand side of $(\ref{eq:P_eps_entropy})$ as already done for the weak residual, cf.~equation~(\ref{eq:b-op-simplified}).
\begin{definition}
Let $\tttfct \in W^{1,\infty}(K_j^n)$ be a non-negative test function,  $ \num{\entropyFlux}$ a numerical entropy flux and let $\num{\vec{u}}\in \pwConst(\stDom)^m$ be a piecewise-constant function. We define the linear local weak entropy dissipation operator as
  \begin{align}
    \label{eq:e-op-def}
    \begin{split}
      &\opE_j^n[\num{\vec{u}}] \, \colon\,  W^{1,\infty}(K_j^n) \to \R,\\
      & \tttfct \mapsto  \Delta x_{j}\left(\entropy\left({\num{\vec{u}}}\,^n_{j}\right) - \entropy\left({\num{\vec{u}}}\,^{n+1}_{j}\right)\right){\avPsi}\,^{n+1}\\
      &\quad\quad\quad\quad+\Delta t^{n+\frac12}\left(\num{\entropyFlux}(\num{\vec{u}}_{j-1}^n, \num{\vec{u}}_{j}^n) \avPsi_{j-\frac12} - \num{\entropyFlux}(\num{\vec{u}}_{j}^n, \num{\vec{u}}_{j+1}^n) \avPsi_{j+\frac12} \right)\\
      &\quad\quad\quad\quad+ \Delta t^{n+\frac12}\entropyFlux\left({\num{\vec{u}}}\,^n_{j}\right)\left( \avPsi_{j+\frac12} - \avPsi_{j-\frac12}  \right),
    \end{split}
  \end{align}
  and the global weak entropy dissipation as $\opE^n[\num{\vec{u}}]\,\colon\, W^{1,\infty}([t^n, t^{n+1}]\times\dom) \to \R$ in the same way as the weak residual on $\stDom$ is defined in  equation~(\ref{eq:def-global-b}).
\end{definition}
 Similarly to Proposition~\ref{prop:B-invariant-under-P} we obtain that
\begin{equation}
  \label{eq:E-invariant-under-P}
  \opE_j^n[\num{\vec{u}}] (\tttfct) = \opE_j^n[\num{\vec{u}}] (\opTestToFem_j^n\tttfct).
\end{equation}
For ease of notation we define the cone of  affine linear functions with non-negative edge averages as $\basisSubPlus(K_j^n) := \{ v \in\basisSub(K_j^n) \,\colon\, \fint_S{v}\, dS \geq 0 \quad \forall  S \in \edges{K_j^n} \}$. In the following we only need to consider $\ttfct \in \basisSubPlus(K_j^n)$ such that  $\opE_j^n[\num{\vec{u}}]( \ttfct) < 0$. We obtain the following bounds, with the same notation as in Propositions~\ref{prop:bjn-estimate} and~\ref{prop:global-bjn},  applying the same technique we have used 
for the weak residual: 
\begin{equation}
  \label{eq:entropy-estimate-loc}
  \inf_{ \tttfct \in W^{1,\infty}(K_j^n),\, \tttfct \geq 0}\frac{\opE_j^n[\num{\vec{u}}]( \tttfct) }{\|\tttfct\|_{W^{1,\infty}(K_j^n)}} \geq 
  C^n \inf_{ \ttfct \in \basisSubPlus(K_j^n)}\frac{\opE_j^n[\num{\vec{u}}] (\ttfct)}{\|\ttfct\|_{\basisSub(K_j^n)}},
\end{equation}
with $C^n$ as in Proposition \ref{prop:global-bjn}, and its global version
    \begin{equation}
      \label{eq:entropy-estimate-glob}
      \inf_{ \|\tttfct\|_{W^{1,\infty}([t^n, t^{n+1}]\times\dom)}=1,\, \tttfct \geq 0}\opE^n[\num{\vec{u}}] (\tttfct) \geq C^n\sum_{j=0}^{J} \inf_{\ttfct\in \basisSubPlus(K_j^n),\, \|\ttfct\|_{W^{1,\infty}(K_j^n)}=1}\opE_j^n[\num{\vec{u}}] (\ttfct) .
    \end{equation}

Next, we present an example in order gain intuition for the scaling of the weak residuals:
\begin{remark}[Local scaling of the weak residual]
  \label{rem:scalings-res}
  Firstly, we consider a specific stationary approximation of a stationary shock in the (scalar) Burgers equations, i.e., $f(u) = \frac12 u^2$,
  \[
    \num{u}_j^n :=
    \begin{cases}
      1 &,\quad j< 0,\\
      0 &,\quad j= 0,\\
      -1 &,\quad j >0.
    \end{cases}   
  \]
This numerical solution is stationary if the Engquist-Osher flux or the Godunov flux is used.
   If the numerical flux from the scheme is also used in the definition of the weak residual the latter satisfies:
  \[
    \left\|\opB_j^n[\num{u}] \right\|_{\mathcal{L}(\basisSub(K_j^n), \R)} =
    \begin{cases}
      0 &,\quad j< 0,\\
      \Delta t^{n+\frac12} \Delta x_0 &,\quad j= 0,\\
      0 &,\quad j >0.
    \end{cases}
  \]
  The residual is concentrated in a single cell and the residual on a time-slab (\ref{eq:estimate-B-global}) satisfies $\|\opB^n[\num{u}]\|_{\mathcal{L}(W^{1,\infty}([t^n, t^{n+1}]\times\dom),\, \R^m)} = O(\Delta x_0 \cdot \Delta t^{n+\frac12})$.
   This, in turn, implies that the so-obtained $\varepsilon$ in the inequality~(\ref{eq:P_eps}), scales as $\varepsilon = O(\Delta x_0)$.
  This seems reasonable, since, arguably, one cannot expect the residual to converge to zero faster than the approximation error.

  It is worthwhile to note that the numerical approximation 
    \[
    \num{u}_j^n :=
    \begin{cases}
      1 &,\quad j\leq 0,\\
      -1 &,\quad j >0,
    \end{cases}   
  \]
  which is also stationary for the Godunov flux and the Engquist-Osher flux, satisfies  $\|\opB^n[\num{u}]\|_{\mathcal{L}(W^{1,\infty}([t^n, t^{n+1}]\times\dom),\, \R^m)} =0$ for all $n$.

  Secondly,  for any numerical solution that is obtained by a Lipschitz continuous and consistent numerical flux  with mesh size $\Delta x := \max \{\Delta x_{j-1}, \Delta x_{j}, \Delta x_{j+1}\}$ one can show that
  \begin{equation}\label{eq:locres}
    \left\|\opB_j^n[\num{u}] \right\|_{\mathcal{L}(\basisSub(K_j^n), \R)}
    \leq
    L \Delta t^{n+\frac12}\left(\Delta t^{n+\frac12} + \Delta x \right) \left( \left|  \num{u}^n_{j-1} - \num{u}^n_{j} \right| 
    + \left|  \num{u}^n_{j} - \num{u}^n_{j+1} \right|  \right)  ,
  \end{equation}
where $L$ is the Lipschitz-constant of the numerical flux so that
\[
\sum_j \left\|\opB_j^n[\num{u}] \right\|_{\mathcal{L}(\basisSub(K_j^n), \R)}
\leq
2L\Delta t^{n+\frac12}\left(\Delta t^{n+\frac12} + \Delta x \right)
TV[\num{u}].
\]
 In this case we obtain that $\varepsilon$ scales as $\varepsilon = O(\Delta x)$, i.e.~with order one in mesh size, same as the numerical scheme. It can be seen from \eqref{eq:locres} that weak-residuals are concentrated in cells adjacent to significant jumps in $\num{u}$.
\end{remark}

Now, we can compute a bound for $\varepsilon$ satisfying both inequalities in $[0, t]\times \dom$ where $t \in (0,T]$ as follows. First, taking the $\ell^\infty \ell^1$-norm of $(\|\opB_j^n[\num{\vec{u}}]\|)_{(j,n)} $ over all cells $K_j^n$ such that $K_j^n \cap ([0,t]\times\dom) \neq \emptyset$ and $(\{t\} \times \dom) \cap \overline{K_j^{N'}} \neq \emptyset$ (i.e.~$N'$ is the largest time-layer of cells in $[0,t]\times\dom$), yields estimates of the weak residual and the entropy dissipation residual in $[0, t]\times \dom$ denoted by
\begin{align}
    &\locBscalar := \max_{n\in\{1,\ldots, N'\}} \frac{1}{\Delta t^{n+\frac12}}\sum_{j=0}^{J} \| \opB_j^n[\num{\vec{u}}] \|_{\mathcal{L}( \basisSub(K_j^n), \R^m)},  \label{eq:beta-eta-1} \\
    &\locEscalar := \max_{n\in\{1,\ldots, N'\}}\frac{1}{\Delta t^{n+\frac12}}\sum_{j=0}^{J} \left| \inf_{\ttfct\in \basisSubPlus(K_j^n),\, \|\ttfct\|_{W^{1,\infty}}=1}\opE_j^n[\num{\vec{u}}] (\ttfct) \right| \label{eq:beta-eta-2}
\end{align} 
where we use the $\ell^{\infty}$-norm on $\R^m$. 
Then, we can determine $\varepsilon$ satisfying 
(\ref{eq:P_eps}) and~(\ref{eq:P_eps_entropy}) as follows
\begin{equation}
  \label{eq:proof-eps-def}
  \varepsilon :=
  \begin{cases}
    0 & \text{if } \num{\vec{u}} \equiv const, \\
    \dfrac{ C\max\left\{ \locBscalar, \locEscalar  \right\} }
      {\sup_{s \in [0,t]}TV[\num{\vec{u}}(s,\,\cdot\,)]} & \text{otherwise},
  \end{cases}
\end{equation}
where $C=\max\left\{3,\, \max_{j\in\{0,\ldots,J\}, n\in\{0,\ldots, N'\}}\sqrt{8+8(c_j^n)^2}\right\}$.

Note that the total variation in (\ref{eq:proof-eps-def}) is zero whenever $\num{\vec{u}}(t,\,\cdot\,) \equiv const$. However, in this case both inequalities (\ref{eq:P_eps}) and~(\ref{eq:P_eps_entropy}) are satisfied trivially, since the weak residual $\locBscalar$ is zero as well and the entropy is constant.


 \subsection{An algorithm for computing bounds on local weak residuals}
  \label{sec:algorithm}
In this section, we present an algorithm for the computation of the previously presented upper bound~
\eqref{eq:estimate-B-global} for  the weak residual and the weak entropy dissipation residual.

  We observe that if the flux used in the definition of $\opB_j^n$ is the same as the one used  in the finite volume scheme, (\ref{eq:finite-volumes}), then    $\opB_j^n[\num{\vec{u}}]((t,x) \mapsto  c) = 0$ for any constant $c\in\R$. Thus, in order to estimate the operator norm $\|\opB_j^n[\num{\vec{u}}]\|_{\mathcal{L}(\basisSub(K^j_n), \R^m)}$ we normalize test functions $\ttfct\in\basisSub(K_j^n)$,
  defined in (\ref{eq:hat-fcts-subspace-t}) as $\ttfct(t,x)=\alpha_1 + \frac{(t^{n+1} - t )}{\Delta t^{n+\frac12}}\alpha_2 + \frac{( x_{j+\frac12} - x)}{\Delta x_{j}}\alpha_3$, by setting  $\alpha_1 = 0$. Note that $\|\diff\ttfct\|_{L^{\infty}}\leq 1$ implies $|\alpha_2| \leq \Delta t^{n+\frac12}$ and $|\alpha_3| \leq \Delta x_j$. Since $\opB_j^n[\num{\vec{u}}]$ is linear,  its norm can be estimated as follows
  \begin{align}
    \label{eq:Bjn-estimate-linear}
    \begin{split}
    \|\opB_j^n[\num{\vec{u}}]\|_{\mathcal{L}(\basisSub(K^j_n), \R^m)} &= \sup_{\|\ttfct \|_{W^{1,\infty}} = 1 } \left|\opB_j^n[\num{\vec{u}}]( \ttfct)  \right| \\
                                                          &\stackrel{\|\diff(\ttfct)\|_{L^\infty} \leq 1}{\leq} \left|\opB_j^n[\num{\vec{u}}]((t,x) \mapsto t)\right|  + \left|\opB_j^n[\num{\vec{u}}]((t,x) \mapsto x)\right|\\
    &  \leq  
   \frac12 \big( \Delta t^{n+\frac12}\big)^2\left|\numflux\left(\num{\vec{u}}^n_{j-1}, \num{\vec{u}}^n_{j}\right)  - \numflux\left(\num{\vec{u}}^n_{j}, \num{\vec{u}}^n_{j+1}\right)  \right|\\
&+\frac12 \Delta x_{j} \Delta t^{n+\frac12}\left|\numflux\left(\num{\vec{u}}^n_{j-1}, \num{\vec{u}}^n_{j}\right) +  \numflux\left(\num{\vec{u}}^n_{j}, \num{\vec{u}}^n_{j+1}\right) - 
2\vec{f}(\num{\vec{u}}_{j}^n)\right|.                                            
    \end{split}
  \end{align}

  We perform  a similar estimate for the entropy dissipation residual operator $\opE_j^n$  keeping in mind that the test functions are in $\basisSubPlus(K_j^n)$, i.e.~their averages are non-negative. We define
  \begin{align}
    \begin{split}
      \opE_1 &:= \opE_j^n[\num{\vec{u}}] \left((t,x) \mapsto  1\right)\\
      &=
        \Delta x_{j}\left(\entropy\left({\num{\vec{u}}}\,^n_{j}\right) - \entropy\left({\num{\vec{u}}}\,^{n+1}_{j}\right) \right) +
        \Delta t^{n+\frac12}\left(\num{\entropyFlux}(\num{\vec{u}}_{j-1}^n, \num{\vec{u}}_{j}^n) - \num{\entropyFlux}(\num{\vec{u}}_{j}^n, \num{\vec{u}}_{j+1}^n)  \right) ,\\
      \opE_2 &:= \opE_j^n[\num{\vec{u}}] \left((t,x) \mapsto  t^{n+1}-t\right) =   \frac12 \big(\Delta t^{n+\frac12}\big)^2\left(\num{\entropyFlux}(\num{\vec{u}}_{j-1}^n, \num{\vec{u}}_{j}^n) - \num{\entropyFlux}(\num{\vec{u}}_{j}^n, \num{\vec{u}}_{j+1}^n)  \right),\\
      \opE_3 &:= \opE_j^n[\num{\vec{u}}]\left((t,x) \mapsto  x_{j+\frac12} - x\right)\\
      & = \frac12 \big(\Delta x_{j}\big)^2\left(\entropy\left({\num{\vec{u}}}\,^n_{j}\right) - \entropy\left({\num{\vec{u}}}\,^{n+1}_{j}\right) \right) +
        \Delta t^{n+\frac12}\Delta x_{j}\left(\num{\entropyFlux}(\num{\vec{u}}_{j-1}^n, \num{\vec{u}}_{j}^n) - \entropyFlux(\num{\vec{u}}_{j}^n)  \right).
    \end{split}
  \end{align}
  Due to linearity of $\opE_j^n[\num{\vec{u}}]$ we obtain 
  \begin{align}
    \label{eq:Ejn-norm-linear}
    \begin{split}
      \inf_{\ttfct \in \basisSubPlus(K_j^n),\, \|\ttfct \|_{W^{1,\infty}} = 1} \opE_j^n[\num{\vec{u}}] &=  \inf_{\ttfct \in \basisSubPlus(K_j^n), \, \|\ttfct \|_{W^{1,\infty}} = 1 }\biggl( \alpha_1\opE_1 + \frac{\alpha_2}{\Delta t^{n+\frac12}}\opE_2  + \frac{\alpha_3}{\Delta x_j}\opE_3\!\!\biggl).\\
      & \geq  \min\{0, \opE_1\}   + \min\{0, \opE_2\} + \min\{0, \opE_3\}.
    \end{split}
  \end{align}
  The last estimate holds due to the fact that  $\|\ttfct\|_{L^\infty} = 1$ implies $|\alpha_1| \leq 1$.
  
  These estimates provide the final ingredients for  Algorithm~\ref{algo:weak-res} that computes $\varepsilon$. 
  \begin{algorithm}[]
    \caption{Weak residual and entropy dissipation estimates}    \label{algo:weak-res}    
    \begin{algorithmic}
      \Function{epsilon}{$\num{\vec{u}}\in {\pwConst(\stDom)}^m$}
        \State $\beta\gets 0$,\, $\eta\gets 0$
        \State $\mathtt{c} \gets  \max_{n=0}^{N } \max_{j=0}^{J} \Delta t^{n+\frac12} / \Delta x_j$, \,\, $\mathtt{C} \gets  \max\{3, \sqrt{8 + 8\mathtt{c}^2}\}$ \Comment{cf. (\ref{eq:estimate-B-global})}

        \ForAll{$n \in 0, \ldots, N-1$} \Comment{cf. (\ref{eq:Bjn-estimate-linear}) and (\ref{eq:Ejn-norm-linear})}
          \State $\displaystyle\quad\beta \gets \max \Big\{ \beta,\,\,  \Big\| \sum_{j =0}^{J} \left|\opB_j^n[\num{\vec{u}}]((t,x)\mapsto t)\right|  + \left|\opB_j^n[\num{\vec{u}}] ((t,x)\mapsto x)\right| \Big\|_{\infty}   \Big\}$
          \State\begin{align*}
                   \eta\gets\max\Big\{\eta,\, \Big\| \sum_{j =0}^{J}  \Big|\min&\{0, \opE_j^n[\num{\vec{u}}] \left((t,x) \mapsto  1\right)\}\\
                   &+ \min\{0, \opE_j^n[\num{\vec{u}}] \left((t,x) \mapsto  t^{n+1}-t\right) \}\\
                   &+\min\{0, \opE_j^n[\num{\vec{u}}]((t,x) \mapsto  x_{j+\frac12} - x) \}\Big| \Big\|_{\infty} \Big\}
                 \end{align*}
        \EndFor{}
        \State\Return{$\mathtt{C}\cdot\max\{\beta, \eta\}\, /\, \max_{n=0}^{N} \textsc{tv}(\num{\vec{u}}(t^n,\, \cdot\, ))$}%
      \EndFunction%
    \end{algorithmic}
  \end{algorithm}
Various quantities such as the averages of $t$ and $x$ can be precomputed. The algorithm requires one swipe through all cells $K_j^n$, i.e. $O(J\cdot N)$ operations.

So far, we have localized the weak residual and the weak entropy dissipation residual, however, these alone do not provide any  error bounds. In the following section we combine our results with  the stability analysis carried out in~\cite{bressanPosterioriErrorEstimates2021a} to obtain the sought.  In addition, we formulate an algorithm for the computation of the estimates of $\opB^n$ and $\opE^n$ and, consequently of the $L^\infty L^1$-error bound of a piecewise-constant  solution.


\section{$L^\infty L^1$ error estimation}
\label{sec:analys-postr-estim}
In the first part of this section, we explain how the  $L^\infty L^1$-error of piecewise-constant approximate solution can be bounded in terms of the computable upper bounds for weak residual and entropy dissipation residual derived in Section~\ref{sec:weak-formulation}. This is based on stability estimates from~\cite{bressanPosterioriErrorEstimates2021a} that depend on certain quantities related to the oscillation of the numerical solution.
We present an algorithm for the computation of these quantities and the overall a-posteriori error estimator in the second part of this section.

\subsection{Computing oscillation bounds}
\label{sec:local-error-estim}
The statements in~\cite[Thm.~3.1 and Thm.~4.1]{bressanPosterioriErrorEstimates2021a} 
provide error estimates for regions where the solution is either smooth or has isolated discontinuities, respectively.
To use Bressan's stability results it is required to compute oscillations in \emph{meso-timeslabs} $[0, \macroT]\times\dom $. We can choose the length of the meso timeslab and arguments from~\cite{bressanPosterioriErrorEstimates2021a}  show that $\macroT \approx \Delta t^{\frac13} $ will lead to good error estimates. In particular, we expect that $\macroT \gg \Delta t $.

As discussed in~\cite{bressanPosterioriErrorEstimates2021a}, our task is to partition the strip $[0,\macroT]\times\dom$ into trapezoids $\trapS_k$, that enclose isolated discontinuities of certain strength and trapezoids  $\trapG_k$, covering the remaining parts of the meso-timeslab, where the solution is mainly ``smooth''. In the following, we distinguish between discontinuities in the numerical solution that are due to the fact that we use the finite-volume approximation but approximate a smooth function and \emph{significant} discontinuities that we believe to approximate discontinuities in the exact solution. 

The idea behind this partitioning is comparing $\num{\vec{u}}$ to solutions of linearized hyperbolic PDEs in ``smooth'' regions whereas in regions near significant discontinuities  $\num{\vec{u}}$ is compared to solutions of Riemann problems. The latter can only be done for sufficiently isolated significant discontinuities. Since the aforesaid isolated significant discontinuities might correspond to different wave types (shocks or contact discontinuities) we  shall refer to them as \textit{surges}. The fact that surges are required to be sufficiently isolated, i.e.~the distance between surges  has to be large enough such that no wave interactions occur in the meso-timeslab  implies that some significant discontinuities of $\num{\vec{u}}$ might end up in the ``smooth'' trapezoids. If this happens it reduces the accuracy of the estimate, however,  this does not affect the scaling of the estimator, see Remark~\ref{rem:error-scalings} for details. 

In order to distinguish surges from ``smooth'' regions, e.g.~piecewise-constant approximations of rarefaction waves,  a lower threshold $\shockThresh>0$ on the surge strength, i.e.~on the maximal absolute difference of the conservative variables, needs to be prescribed.   Following the construction in~\cite[Section 4]{bressanPosterioriErrorEstimates2021a}, a surge is enclosed in a narrow strip $\{ (t,x)\,\colon\, t \in [0,\macroT],\, |x-\gamma(t)| < \delta  \}$ around the approximate surge curve $\gamma(t) = x_0 + \lambda t$ originating at some footpoint $x_0 \in \dom$. We choose the strip-width $\delta$ of order $\delta = O(\varepsilon^{\frac23})$ which will ensure correct asymptotic behavior later on. The narrow strip is contained in a trapezoid of width at least $2\left(\delta + \varepsilon^{\frac13}\right)$.  Suppose we have localized $J_S$ surges in a meso-timeslab (details on this will be given in Algorithm~\ref{algo:surge-detect}). Then we define the trapezoid surrounding the $k$-th surge originating at $x_{0,k}$ as a convex hull
\begin{align}
  \label{eq:shock-trapezoids}
  \begin{split}
    &\trapS_k := \text{conv}\left\{(0, a'_k),\, (0, b'_k),\, (\macroT, a_k),\, (\macroT, b_k)   \right\}, \text{ where: }\\
    &a_k = x_{0,k} + \lambda^-\macroT - \delta_k - \varepsilon^{\frac13},\quad
    b_k = x_{0,k} + \lambda^+\macroT + \delta_k + \varepsilon^{\frac13}\\
    &a'_k = a_k - \lambda^+\macroT,\quad b'_k = b_k - \lambda^-\macroT,
  \end{split}
\end{align}
for $k\in \{0,\ldots, J_S\}$. Let $\lambda^+$ and $\lambda^-$ denote the maximal and minimal  speed of the characteristics  in $[0,\macroT]\times\dom$,  respectively.
The remaining parts of the meso-timeslab, $[0,\macroT]\times\dom\,\, \backslash\,\,  \bigcup_{k=0}^{J_S} \trapS_k$, are covered, in accordance with~\cite[Eqs. (3.1)--(3.5)]{bressanPosterioriErrorEstimates2021a}, by 
\begin{align}
  \label{eq:trapezoids-def}
  \begin{split}
    &\trapG_k := \text{conv}\left\{(0, a'_k),\,(0, b'_k),\,(\macroT,\, a'_k + \macroT\lambda^{+}+\varepsilon^{\frac23}),\, (\macroT,\, b'_k + \macroT\lambda^{-}-\varepsilon^{\frac23})   \right\},
  \end{split}
\end{align}
where $a'_k < a'_{k+1} < b'_k < b'_{k+1}$ for all $k\in\{0,\ldots, J_G\}$, are chosen such that $b'_k - a'_k > 2\macroT(\lambda^+ - \lambda^-)$. The last condition ensures that each point in the meso-timeslab is contained in at most two smooth trapezoids, and, if there are only finitely many surges in the exact solution, the number of surge trapezoids should be bounded uniformly in $\Delta x$ (by the number of surges in the exact solution). The slopes of the trapezoids $\trapG_k$ are given by the maximal and minimal characteristic speeds adjusted by $\varepsilon^{\frac23} /  \tau \approx \varepsilon^{\frac13}$ and we have     $[0,\macroT]\times\dom = \bigcup_{k=0}^{J_S} \trapS_k \,\cup\,   \bigcup_{k=0}^{J_G} \trapG_k $. 

In the trapezoids covering ``smooth'' regions, the supremum of oscillations  is defined as
\begin{equation}
  \label{eq:kappa-def}
  \osc := \max_{k\in \{0,\ldots, J_G\}}  \sup_{(t,x),\,(s,y) \in \trapG_k} \left| \num{\vec{u}}(t,x) - \num{\vec{u}}(s,y) \right|.
\end{equation}

The final ingredient required for the $L^\infty L^1$-error estimators are the oscillations of the solution in the vicinity of a surge. For  surge-trapezoids, $\trapS_k$, it is expected that oscillations outside of the narrow strip of width $\delta_k$ surrounding the $k$-th surge curve are of the same order of magnitude as the oscillations in ``smooth'' regions. Thus, we define the left and right sub-trapezoids of $\trapS_k$ excluding the strip around the $k$-th surge as
\begin{align}
  \label{eq:left-right=doms}
  \begin{split}
    &\trapS_k^{l} := \left\{ (t,x)\in \trapS_k\,\colon\, a'_k + \lambda^+t \leq x \leq x_{0,k} - \delta_k + \lambda t \right\},\\
    &\trapS_k^{r} := \left\{ (t,x)\in \trapS_k\,\colon\, x_{0,k} + \delta_k + \lambda t \leq x \leq b'_k + \lambda^-t  \right\},\\ 
  \end{split}
\end{align}
cf. \cite[Eqs. (4.5)-(4.6)]{bressanPosterioriErrorEstimates2021a}, and the oscillations in $\trapS_k$ as
\begin{equation}
  \label{eq:shock-oscillations}
  \osc_k' := \max\left\{\sup_{(t,x),\,(s,y) \in \trapS_k^{l}} \left| \num{\vec{u}}(t,x) - \num{\vec{u}}(s,y) \right|, \sup_{(t,x),\,(s,y) \in \trapS_k^{r}} \left| \num{\vec{u}}(t,x) - \num{\vec{u}}(s,y) \right|\right\}.
\end{equation}
\begin{remark}
Note that we do not require precise surge curves. Indeed, we exclude strips  of width  $\delta_k = O(\varepsilon^{2/3})$ around approximate surge curves $\gamma_k$ when computing oscillations in surge trapezoids $S_k$. This allows us to handle uncertainty of surge positions due to numerical dissipation as well as the fact that, in general, surge curves do not have constant speed. In the following subsection we present algorithms for surge detection and determination of widths of strips around surges $\delta_k$.
\end{remark}

The above definitions allow us to state the following local in time a-posteriori error estimate, summarizing~\cite[Thm.~3.1]{bressanPosterioriErrorEstimates2021a} and~\cite[Thm.~4.1]{bressanPosterioriErrorEstimates2021a}.
\begin{theorem}\label{thr:time-slab}
  Let $\opSemiG \,\colon\, [0,\infty) \times  \Xi \to \Xi$ define a Lipschitz-semigroup of entropy weak solutions of the hyperbolic conservation law~(\ref{eq:pde}) in the domain of the semigroup $\Xi \subset L^1(\R;\,\R^m)$ of functions with small total variation. Given a threshold on the surge strength, $\shockThresh_0 > 0$ we define
  \begin{equation}
    \label{eq:shock-strength-bound}
    C_0 := \frac{\shockThresh_0}{ \min_{k \in \{0,\ldots, J_S\}}\left( \varepsilon / \delta_k + \delta_k / \varepsilon^{\frac13} + 2\osc'_k \right)^{\frac13}}.
  \end{equation}
   Let the meso-timeslab $[0,\tau] \times \dom $ be partitioned  into trapezoids according to~(\ref{eq:shock-trapezoids}) and (\ref{eq:trapezoids-def}). Then, there exist constants $C_S$ and $C_G$ depending on $C_0$ such that the error at time $\tau$ of the solution is estimated as
  \begin{equation}
    \label{eq:error-estimate-bressan}
    \max_{0\leq t \leq \macroT}\left\|\num{\vec{u}}(t, \,\cdot\,) - \opSemiG_t \vec{u}_0 \right\|_{L^1(\dom )} \leq C_S\estim_S + C_G\estim_G
  \end{equation}
  where
  \begin{equation}
    \label{eq:error-estimate-bressan-shock}
    \estim_S :=  \macroT\sum_{k=0}^{J_S} \left( \frac{\varepsilon}{\delta_k} + \frac{\delta_k}{\varepsilon^{\frac13}}  + 2\osc'_k  \right),\quad
    \estim_G :=  \left( \macroT  + \osc  \right) \sup_{t \in [0,\macroT]}TV[\num{\vec{u}}(t,\,\cdot\,)]\, \varepsilon^{\frac13}.
  \end{equation}
\end{theorem}
This theorem provides a basis for constructing adaptive mesh refinement methods for time-marching schemes and, more importantly, serves as the basis for the  error estimate over the whole space-time domain $\stDom$ presented right after a few remarks. 

\begin{remark}\label{rem:main-thm} \ \newline  \begin{enumerate}
  \item The discussion in~\cite[Rem. 5.1]{bressanPosterioriErrorEstimates2021a} sheds light on the expected asymptotic behavior of the error bound~(\ref{eq:error-estimate-bressan-shock}). The variety of choices in the construction of the stability estimates implies that a rigorous analysis of the asymptotics cannot be performed  easily. However, the same variety allows tuning the estimates in a heuristic manner in hope of achieving sharper bounds. This considerations led to the choices of the parameters that are incorporated in the statement of Theorem~\ref{thr:time-slab} and are expected to provide error bounds  approaching zero with $\varepsilon^{\frac23}|\log\varepsilon|$. In the following section we will perform numerical experiments to investigate the error bound asymptotics.
  \item Many choices in the computation of the error estimator, e.g.~the precise choice of  the trapezoidal domain decomposition or the precise values of $\delta_k$, influence how sharp the estimate in Theorem~\ref{thr:time-slab} is. However, as long as the scaling of these choices is appropriate, e.g.~as long as $\delta_k$ scales with $\varepsilon^{\frac23}$, these choices do not affect the overall asymptotics.  
  \item Prescribing a threshold $\shockThresh$ on the surge strength allows us to select the sharper error bound in~\cite[Thm. 4.1]{bressanPosterioriErrorEstimates2021a}.
  \item The estimate given in  Theorem~\ref{thr:time-slab} is local in time in the sense that it quantifies errors on a meso-timeslab. Implicitly, there is a stronger locality of the error bound  in the sense that diminishing local residual contributions  of a particular space-time cell (e.g.~by refining it in an adaptive mesh setting) reduces the overall error bound. However, due to the nature of non-linear hyperbolic conservation laws, in particular the lack of coercivity, obtaining error bounds for specific cells or showing efficiency of a-posteriori error estimates seems out of reach.
  \end{enumerate}
\end{remark}

We summarize the analytic results of this work with an estimate of the $L^\infty L^1$-error of a finite volume solution on the whole space-time domain $\stDom$. To this end, we define $\varepsilon_T$ as in~(\ref{eq:proof-eps-def}), however incorporating all cells in $\stDom$. We partition the domain into timeslabs of size $\tau\approx\varepsilon_T^{\frac13}$ and denote quantities in the $\mu$-th timeslab by the superscript $^\mu$. Thus, we denote the ``smooth'' oscillations in the $\mu$-th timeslab as $\osc^\mu$ and the oscillations in the $k$-th surge-trapezoid by $(\osc_k')^{\mu}$, as in~(\ref{eq:kappa-def}) and~(\ref{eq:shock-oscillations}), respectively. Furthermore,  the maximal oscillation of all surge-trapezoids is denoted by $\osc'_{\max} := \max_{k,\mu} (\osc_k')^{\mu}$ and similarly the maximal width of a strip around a surge by $\delta_{\max}:= \max_{k,\mu} (\delta_k)^{\mu}$. Summing up the estimates $\estim_S$ and $\estim_G$, as well as pulling $\osc'_{\max}$ and $\delta_{\max}$ out of the sum yields the main theorem. 

\begin{theorem}
  \label{thr:main}
  Let $\opSemiG$ define a Lipschitz-semigroup of entropy weak solutions of the hyperbolic conservation law~(\ref{eq:pde}) and $\shockThresh > 0$ a lower threshold on the surge strength. Let the space-time domain $\stDom $ be uniformly partitioned into $M+1$ timeslabs of size $\tau\approx\varepsilon_T^{\frac13}$, i.e. $[0,T] = \bigcup_{\mu=0}^{M} [\mu\tau, (\mu+1)\tau] $. Let the $\mu$-th timeslab contain $J^\mu_S$ surges  and be  partitioned into trapezoids $\trapS^\mu_k$ and $\trapG^\mu_k$ as defined in~(\ref{eq:shock-trapezoids}) and (\ref{eq:trapezoids-def}), respectively. Then, there exist constants $C_S$ and $C_G$ depending on $C_0$, cf.~(\ref{eq:shock-strength-bound}),  such that the $L^\infty L^1$-error at time $T$ of the piecewise-constant solution $\num{\vec{u}} \in \pwConst( \stDom )$ is estimated as
  \begin{equation}
      \label{eq:main-estim}
       \max_{0\leq t \leq T}\|\num{\vec{u}}(t,\,\cdot\,) - \opSemiG_t\vec{u}_0 \|_{L^1(\Omega)}\leq  C_S \estim_S^T + C_G\estim^T_G
    \end{equation}
    where
    \begin{equation}
      \label{eq:main-estim-es-eg}
      \estim_S^T :=  \left( \varepsilon_T^{\frac13}\osc'_{\max} + \delta_{\max}   \right)\sum_{\mu=0}^M J_{S}^{\mu} , \quad \estim_G^T:= \varepsilon_T^{\frac13}\left(T + \sum_{\mu=0}^M \kappa^{\mu}  \right).
    \end{equation}
  \end{theorem}
  \begin{remark}
    \label{rem:constants-main-thm}
    Unfortunately, it is not feasible to compute either of the constants $C_S$ and $C_G$ since they depend, among other quantities, on the Lipschitz constant of the semigroup $\opSemiG$  and on constants that emerge from the implicit function theorem used to estimate the jump strength.
  \end{remark}

  \begin{remark}[Scaling of the error estimator]
    \label{rem:error-scalings}
    \begin{enumerate}
    \item The arguments  in~\cite[Rem. 5.1]{bressanPosterioriErrorEstimates2021a} assert that, if the numerical solution approximates an exact solution that is piecewise Lipschitz with finitely many discontinuities, then the oscillations in trapezoids containing surges fulfill $\osc'_{k} = C'\tau$ for some constant $C'>0$. This is due to the fact that $\osc'_k$ measures oscillations excluding the surge curve, where the solution is assumed to be Lipschitz continuous.   If no wave interactions are present, then $\estim_G$ behaves as $\varepsilon^{\frac13}|\log\varepsilon|$, as mentioned in Remark~\ref{rem:main-thm}, and, since we assumed $\delta_k = O(\varepsilon^{\frac23})$, the heuristic expectation of the overall behavior of the error estimator is
      \begin{equation}
        \label{eq:error-asymptotics}
        C_S \estim_S^T + C_G\estim^T_G = O(1)\cdot\varepsilon^{\frac13}|\log\varepsilon| = O(1)\cdot (\max_n \Delta t^{n+\frac12})^{\frac13}\left| \log \max_n \Delta t^{n+\frac12} \right|.
      \end{equation}      
    \item   If a numerical solution features a surge that is smeared out to the extent that it remains undetected, then its contribution to the oscillation in a ``smooth'' trapezoid will be small (similar to the oscillations of other smooth trapezoids). Thus, this will increase the amount by which the error estimator overestimates the error but it will not reduce the convergence order of the error estimator.  
      
    \item The proposed trapezoid partition of the domain enforces a distance of $O(\varepsilon^{\frac13})$ between surges and, thus, points in the space-time domain where surges interact are necessarily covered by smooth trapezoids $G_k$. This implies that the oscillations in these $G_k$ will be large. The discussion in~\cite[Rem. 5.1]{bressanPosterioriErrorEstimates2021a} shows that if the exact solution has only finitely many surges then the sum of length of time-slabs with undetected surge interactions is
      \[
        O(1)\cdot\varepsilon^{\frac13}\cdot[\text{total nr.~of surge interactions}]  = O(\varepsilon^{\frac13}).
      \]
      Thus, the undetected surge-interactions do not affect the asymptotics of the right-hand side of~(\ref{eq:main-estim}). 
    \end{enumerate}
  \end{remark}

\subsection{Algorithm for computing suitable trapezoids}
 In order to make the error estimator \eqref{eq:main-estim} fully computable, it remains to formulate an algorithm for the decomposition of meso-timeslabs into trapezoids containing ``smooth'' parts and trapezoids containing surges. 

In the  algorithms presented in this section we employ the following data structures. A linked $\mathtt{list}$ is initialized by a  set-like notation or as an empty  list if no arguments are provided. The $k$-th element of a list $\mathtt{L}$ is denoted by the subscript notation $\mathtt{L}_k$. Trapezoids are represented via the $\mathtt{trpz}(I_1, I_2)$-structure where $I_1=[a',b']$ and $I_2=[a,b]$ are the bottom and the upper interval of the trapezoid, respectively. Corners of a $\mathtt{trpz}$-object $\mathtt{T}$ are denoted by  $\mathtt{T}.a'$, $\mathtt{T}.b'$, $\mathtt{T}.a$ and $\mathtt{T}.b$. Trapezoids containing strong discontinuities are stored via $\mathtt{surgetrpz}(S^l, \, S^r)$, where $S^l$ and $S^r$ are the parts of the trapezoid  away from the surge, as defined in~(\ref{eq:left-right=doms}). These two trapezoidal regions are sufficient to define a trapezoid enclosing a surge and, moreover, we are interested in the oscillations in $S^l$ and $S^r$ only. Furthermore, we define a function $\mathtt{inb}(x)$ that projects $x$ to the domain $\Omega$ that we employ to avoid trapezoids leaving the domain.

The function $\textsc{surge-trpzs}$ searches for sufficiently isolated surge-footpoints  in the lowest timelevel of each meso-timeslab. Subsequently, the intersection of the  cone of extreme characteristic speeds emanating from each potential surge-footpoint with the highest timelevel  of the meso-timeslab  is scanned for corresponding surges.

  In order to reliably tag cells containing jumps the function $\textsc{detect-jumps}$ implements multiresolution analysis techniques  based on wavelets~\cite{dahmenMultiresolutionSchemesConservation2001,muller2002adaptive}.  In principle, any appropriate discontinuity indicator could be employed, for instance the entropy-based indicator introduced in~\cite{sempliceAdaptiveMeshRefinement2016}.

  In order to identify the trapezoidal  decomposition of the domain $\stDom$, that is required in Theorem~\ref{thr:main}, we proceed by iterating through meso-timeslabs of size $\tau\approx\varepsilon^{\frac13}$. In each timeslab, strips of width $\varepsilon^{\frac23}$ containing surges are tagged. If we find no or more than one jump regions in this cone at the highest time level no surge curve is created. This may mean that certain significant discontinuities are not tagged as surges causing some mild overestimation that has been discussed in the second and third point of Remark~\ref{rem:error-scalings}.

  Since the exact positions of  surge-footpoints, $x_{0,k}$, are unknown, we set $x_{0,k}$ to the center of the found  surge-region in the lowest time-level of the meso-timeslab. Thus, if a surge  is found the next step is determining the width of the strip around the $k$-th surge $\delta_k$.  To this end, we divide the strip of width $\delta_k$ in  two parts: one to the left and one to the  right of the surge with widths  $\delta_k^l$ and $\delta_k^r$, respectively. Each initial width $\delta_k^{l,r}$ is set to a slightly smaller value than the width of the sub-trapezoids $S_k^{l} $ and $ S^r_k$, i.e.~$\delta^{l,r}_k = \varepsilon^{\frac13} - \varepsilon^{\frac23} $. We save the maximal and minimal value of the solution in the (thin) trapezoids $S_k^{l,r}$. Subsequently, we reduce each $\delta_k^{l,r}$ by $\varepsilon^{\frac23}$ (and by this increase the area of $S_k^{l,r}$) until either the oscillations $S_k^{l,r}$ computed by~\eqref{eq:shock-oscillations} reach a value larger or equal $\tau$ or $\delta_k$ reaches $\varepsilon^{\frac23}$. Here, we
  include oscillations of $\num{\vec{u}}$ in all cells that have non-empty intersection with the trapezoid. In each step, it is required to calculate the maximal and minimal values only in the added area of $S_k^{l,r}$. Hence, the cost of this procedure is the same as the cost of calculating the oscillations in $S_k^{l,r}$. This procedure, described in detail in Algorithm~\ref{algo:surge-detect}, ensures that  large oscillations concentrate in a smallest possible strip of width $\delta_k$ keeping oscillations in $S_k^{l,r}$ small.
  
  \begin{algorithm}[!h]
    \caption{Surge detection}\label{algo:surge-detect}
    \begin{algorithmic}
      \Function{cstr-surge-trpz}{$\tau,\, \mathtt{surge} ,\, \delta^l,\, \delta^r,\, \lambda^-,\, \lambda^+ $} \Comment{{\small  cf. (\ref{eq:shock-trapezoids}) and  (\ref{eq:left-right=doms})}}
        \State $[j_1,\, j_2,\, k_1,\, k_2] \gets \mathtt{surge} $
        \State $x_0 \gets  \frac12(x_{j_1 - \frac12} + x_{j_2 + \frac12})$ 
        \State $a\gets  \textsc{inb}(x_0 + \lambda^-\tau - \delta^l - \varepsilon^{\frac13}),\quad b\gets  \textsc{inb}(x_0 + \lambda^+\tau + \delta^r + \varepsilon^{\frac13})$ 
        \State $a' \gets \textsc{inb}(a - \lambda^+\tau),\quad \, b' \gets \textsc{inb}( b - \lambda^- \tau)$
        \State $\lambda \gets    \left(\frac12(x_{k_1 - \frac12} + x_{k_2+\frac12}) - x_0\right) /\tau $ 
        \State $S^l \gets  \mathtt{trpz}\left([a',\, x_0 - \delta^l],\, [a,\, x_0 - \delta^l + \lambda\tau  ]\right) $
        \State $S^r \gets   \mathtt{trpz}\left( [x_0 + \delta^r,\, b'],\, [x_0 + \delta^r +\lambda \tau,\, b]\right)$
        \State\Return $\mathtt{surgetrpz}(S^l,\, S^r)$
      \EndFunction      
      \Function{surge-trpzs}{$t^n, \, t^{n+l},\, \num{\vec{u}}, \, \shockThresh_0, \, \lambda^-,\,\lambda^+$}
        \State $\mathtt{surges} \gets  \mathtt{list}(\,\,) $, $\mathtt{deltas} \gets  \mathtt{list}(\,\,)$, $\mathtt{oscs} \gets  \mathtt{list}(\,\,)$, $\mathtt{S} \gets \mathtt{list}(\,\,)$
        \State $\tau \gets t^{n+l} - t^n$, \quad $\mathtt{sd} \gets  (\lambda^{+} - \lambda^{-})\tau $ \Comment{{\small minimal dist. between surge-trapezoids}}
        \State $\mathtt{jumps^{\text{b}}} \gets  \textsc{detect-jumps}(t^n,\, \dom,\, \num{\vec{u}}, \, \sigma_0)$
        \State $\textsc{filter}(\mathtt{jumps^{\text{b}}} ,\, \textsc{dist}(\mathtt{jumps}^{\text{b}}_k, \mathtt{jumps}^{\text{b}}_{k+1}) < \mathtt{sd} \text{ for } k \in0,\ldots,\mathtt{jumps}^{\text{b}}.size() - 2) $
        \For{$ (j_1,\, j_2) \in \mathtt{jumps}^{\text{b}}$}
          \State $\mathtt{jumps^{\text{t}}} \gets \textsc{detect-jumps}(t^{n+l},\, [x_{j_1} + \lambda^-\tau,\, x_{j_2} + \lambda^+\tau] ,\, \num{\vec{u}})$
          \If{$\mathtt{jumps^{\text{t}}}.size() = 1$ }
            \State $(k_1, k_2) \gets  \mathtt{jumps}^{\text{t}}_0$,\quad $\mathtt{surges}.append( [(j_1,\,j_2),\, (k_1,\,k_2)])$
            \State $\mathtt{deltas} \gets  \textsc{max}(x_{j_2} - x_{j_1},\, x_{k_2} - x_{k_1})$
          \EndIf          
        \EndFor
        \For{$ \mathtt{s} \in \mathtt{surges} $}
          \State $\delta \gets  2(\varepsilon^{\frac13} - \varepsilon^{\frac23}) $, \quad $\delta^l \gets  \frac12\delta$, \quad $\delta^r \gets  \frac12\delta$
          \State $S_{old} \gets  \textsc{cstr-surge-trpz}(\tau,\, \mathtt{s},\, \delta^l,\, \delta^r,\, \lambda^-,\, \lambda^+)$
          \State $\vec{m}^{l,r} \gets \textsc{min}(\num{\vec{u}}\big|_{S^{l,r}_{old}})$,\quad $\vec{M}^{l,r} \gets  \textsc{max}(\num{\vec{u}}\big|_{S^{l,r}_{old}})$
          \While{$|\vec{M}^{l,r} - \vec{m}^{l,r}| \leq \tau $ and $ \delta^{l} + \delta^r \geq \varepsilon ^{\frac23}$} \Comment{Determine smallest  $\delta$ }
            \State $\delta^{l,r} \gets \delta^{l,r} - 2\varepsilon^{\frac23}$
            \State $S_{new} \gets  \textsc{cstr-surge-trpz}(\tau,\,\mathtt{s},\, \delta^l,\, \delta^r,\, \lambda^-,\, \lambda^+)$ 
              \State $\vec{m}^{l,r} \gets \textsc{min}(\vec{m}^{l,r}, \textsc{min}(\num{\vec{u}}\big|_{S_{new}^{l,r} \backslash S_{old}^{l,r}}))$
              \State $\vec{M}^{l,r} \gets \textsc{max}(\vec{M}^{l,r}, \textsc{max}(\num{\vec{u}}\big|_{S_{new}^{l,r} \backslash S_{old}^{l,r}}))$
            \State $S^{old} \gets S^{new}$ 
          \EndWhile
          \State $\mathtt{S}.append\left(\mathtt{surgetrpz} \left( S^l, S^r  \right)\right)$, \, $\mathtt{oscs}.append(\max\{ |\vec{M}^l - \vec{m}^l|, \, |\vec{M}^r - \vec{m}^r|\} )$
        \EndFor
        \State\Return $\mathtt{S}$, $\mathtt{oscs}$
      \EndFunction
    \end{algorithmic}
  \end{algorithm}
The remaining ``smooth'' areas of the meso-timeslab are covered with trapezoids such that each cell is contained in at least one and at most two trapezoids.  The domain-partitioning algorithm for a meso-timeslab is summarized in Algorithm~\ref{algo:dom-part}.

The algorithm requires two sweeps in the spatial direction of the solution, i.e.~for each meso-timeslab partition $O(J)$ operations are required. Algorithm~\ref{algo:dom-part} returns trapezoid corners of $\mathtt{S}$ and $\mathtt{G}$  as points in the domain $\stDom$ from
which we can retrieve the cells as well as a list of oscillations $\osc'_k$ for each surge-trapezoid. This information is used  in function $\textsc{osc}$ that computes oscillations according to~\eqref{eq:kappa-def}. Also here, we
 include oscillations of $\num{\vec{u}}$ in all cells that have non-empty intersection with the trapezoid.
  Its output is used in the final Algorithm~\ref{algo:main} that   computes the error bounds (\ref{eq:main-estim}).  The complexity of the algorithm is bounded by $O(J\cdot N)$, i.e.~the total number of space-time cells, since each cell may appear in at most two trapezoids.
 \begin{algorithm}[H]
  \caption{Error estimation}\label{algo:main}
  \begin{algorithmic}
    \Function{error-estimator}{$\num{\vec{u}}\in {\pwConst([0,\, T]\times\dom)}^m$, $\shockThresh_0$}
      \State $\varepsilon \gets \textsc{epsilon} (\num{\vec{u}})$,  $\kappa \gets  0$, $\kappa' \gets  0$, $J_S \gets 0$, $\delta_{max} \gets 0$
      \State $\mathtt{taus} \gets \textsc{concatenate}( 0, \,  \mathtt{list}( t^n\,\colon\, t^{n-1} < \mu\varepsilon \leq t^n \leq T,\, \mu = 1,2,\ldots)) $
      \For{$\mu \in 0,\ldots,\mathtt{taus}.size() - 1$}
        \State $\mathtt{S},\, \mathtt{Sosc},\, \mathtt{G} \gets  \textsc{meso-slab-partition}(\num{\vec{u}}( [t^\mu, t^{\mu+1} ], \,\cdot\,), \varepsilon, \shockThresh_0) $
        \State $\kappa \gets  \kappa + \max_{\mathtt{g}\in\mathtt{G}}\textsc{osc}(\mathtt{g})$,\, $\kappa' \gets  \mathtt{max}\left\{ \kappa', \mathtt{max}(\mathtt{Sosc})\right\}$
        \State $J_S \gets J_S + \mathtt{S}.size()$,\, $\delta_{max} \gets \mathtt{max}(\delta_{max}, \delta(\mathtt{S}))$
      \EndFor
      \State\Return $\estim_S = \left( \varepsilon^{\frac13}\osc'_{\max} + \delta_{\max}  \right)J_S$, $\estim_G = \varepsilon^{\frac13}\left( T + \kappa \right)$
    \EndFunction
  \end{algorithmic}
\end{algorithm}
\vspace{1cm}


\section{Numerical Results}
\label{sec:numerical-results}

We conduct three numerical experiments to investigate the scaling behavior of the a-posteriori error estimators derived in the previous chapters.  The first two deal with the $p$-system
\begin{align}
  \begin{split}
    &\rho_t + q_x = 0,\\
    &q_t + \left( \frac{q^2}{\rho} + p(\rho)  \right)_x = 0, 
  \end{split}
\end{align}
where $\rho$ denotes the density, $q = \rho v$ the momentum and $p(\rho) = C\rho^{\gamma}$ the pressure with $C=1$ and $\gamma=\frac75$. We consider two Riemann problems: one where the solution consists of two rarefaction waves and one where we have a rarefaction and a shock.  For the third test case we consider a solution to the Burgers equation
\begin{equation}
  \label{eq:burgers}
  u_t + \left( \frac12 u^2 \right)_x = 0,
\end{equation}
 containing a single shock whose speed is rapidly changing. For each case we study $\varepsilon$ from (\ref{eq:proof-eps-def}) as well as  $\estim_S$ and $\estim_G$  from (\ref{eq:main-estim-es-eg}) for varying grid resolution. The quantities are computed by means of Algorithm~\ref{algo:main}. 

 In all three test cases we employ the local Lax-Friedrich numerical flux, i.e.
 \[
   \num{f}(\num{\vec{u}}_L,\, \num{\vec{u}}_R ) = \frac12\left(\vec{f}(\num{\vec{u}}_L) + \vec{f}(\num{\vec{u}}_R)  \right) - \frac{1}{2} \lambda_{\max} \left(\num{\vec{u}}_R - \num{\vec{u}}_L  \right)
 \]
 with $\lambda_{\max} = \max\{\| \lambda( \num{\vec{u}}_L) \|_{\infty},\, \|  \lambda(\num{\vec{u}}_R)  \|_{\infty} \} $,
 where $\lambda(\vec{u})$ denotes the set of eigenvalues of $D\vec f(\vec u) $, and the corresponding numerical entropy flux
 \[
\num{\entropyFlux}(\num{\vec{u}}_L,\, \num{\vec{u}}_R ) =  \frac12\left(q(\num{\vec{u}}_L) + q(\num{\vec{u}}_R)  \right) - \frac{1}{2} \lambda_{\max} \left( e\left( \num{\vec{u}}_R \right) - e \left( \num{\vec{u}}_L\right)  \right).
 \]

The CFL-number is set to 0.9 and the spatial grid consists of $2\cdot 2^L$ cells uniformly discretizing $\Omega=(-5,5)$. Here, $L$ denotes the level of refinement and, thus, the coarsest possible grid, for $L=0$, contains two cells. The boundary ghost cells for all cases are set to constant continuations of the values in their respective neighbors. 
We monitor the scaling of several quantities with respect to the level $L$ and for a generic quantity $e$ we supplement the dataset $e_L$ with empirical orders of convergence $\text{EoC} := -\log_2(e_{L+1} / e_L)$.  

For the shock detection we employ the Daubechies wavelets and set the minimal surge strength to $\sigma=0.1$. The error estimator is implemented in the \textsc{multiwave}-framework~\cite{gerhardWaveletFreeApproachMultiresolutionBased2021}.

\subsection*{The $p$-system: two rarefactions}
\label{sec:smooth-case}

For the first case we start our scheme at $t=0.5$ with the exact solution of the Riemann problem to the initial data
\begin{equation}
  \vec{u}(0,\, x) =
  \begin{cases}
    (1,\, -2)^T &\text{ if } x \leq 0, \\
    (1,\, 2)^T &\text{ if } x > 0. \\
  \end{cases}
\end{equation}
The solution to the Riemann-problem (RP) can be computed  by solving for a root of a scalar function, cf.~\cite{toroRiemannSolversNumerical2009}. In this case, the solution consists of a left-moving and a right-moving rarefaction wave. We run the computation from $t=0.5$ to $t=1.0$. During this time interval the exact solution contains two fully developed rarefaction waves, i.e.~it is continuous but not smooth due to the kinks at the rarefaction waves. 

  Table~\ref{tab:sine} summarizes the results of  Algorithms~\ref{algo:weak-res} and~\ref{algo:main}. The error estimator for the surge part is omitted since the numerical solution does not feature any surges and, thus, $\estim_S = 0$ for all levels $L$.
\begin{table}[!htb]
  \begin{center}
    \begin{tabular}{c|cc|c|cc|cc}
      $L$ & $\varepsilon$ & EoC & $\varepsilon^{\frac13}$ & $\estim_G$ & EoC & $L^{\infty}L^1$-err & EoC\\
      \hline
     7 & 0.19297 &  & 0.57787   & 0.87028 &  & 0.49868 & \\
8 & 0.09584 & 1.01 & 0.45764  & 0.78796 & 0.14 & 0.31273 & 0.67\\
9 & 0.04780 & 1.00 & 0.36291  & 0.66305 & 0.25 & 0.19232 & 0.70\\
10 & 0.02388 & 1.00 & 0.28797 & 0.56602 & 0.23 & 0.11582 & 0.73\\
11 & 0.01194 & 1.00 & 0.22857 & 0.47899 & 0.24 & 0.06852 & 0.76\\
12 & 0.00597 & 1.00 & 0.18144 & 0.38769 & 0.31 & 0.03987 & 0.78\\
    \end{tabular}
  \end{center}
  
  \caption{Error estimator results for the two-rarefactions-RP for the $p$-system.}
  \label{tab:sine}
\end{table}
The empirical order of convergence of $\varepsilon$ is  almost exactly 1 throughout the refinement levels. The error estimator for the smooth part of the solution, $\estim_G$,  approaches the theoretically predicted order of $\frac13$.

\subsection*{Riemann problem for the $p$-system: rarefaction and shock}
\label{sec:sod-problem-with}
We consider the Riemann initial data
\begin{equation}
  \vec{u}(0,\, x) =
  \begin{cases}
    (0.15,\, 0)^T &\text{ if } x \leq 0, \\
    (0.1,\, 0)^T &\text{ if } x > 0. \\
  \end{cases}
\end{equation}
The resulting solution consists in a left-going rarefaction wave and a right-going shock wave emerging from $x=0$, respectively. We simulate the solution from $t=0$ to $t=1.5$ and the space-time plot as well as the trapezoidal decomposition are shown in Figure~\ref{fig:sod-plot-god}. 
\begin{figure}[!htb]
  \centering
  \includegraphics[width=0.5\textwidth]{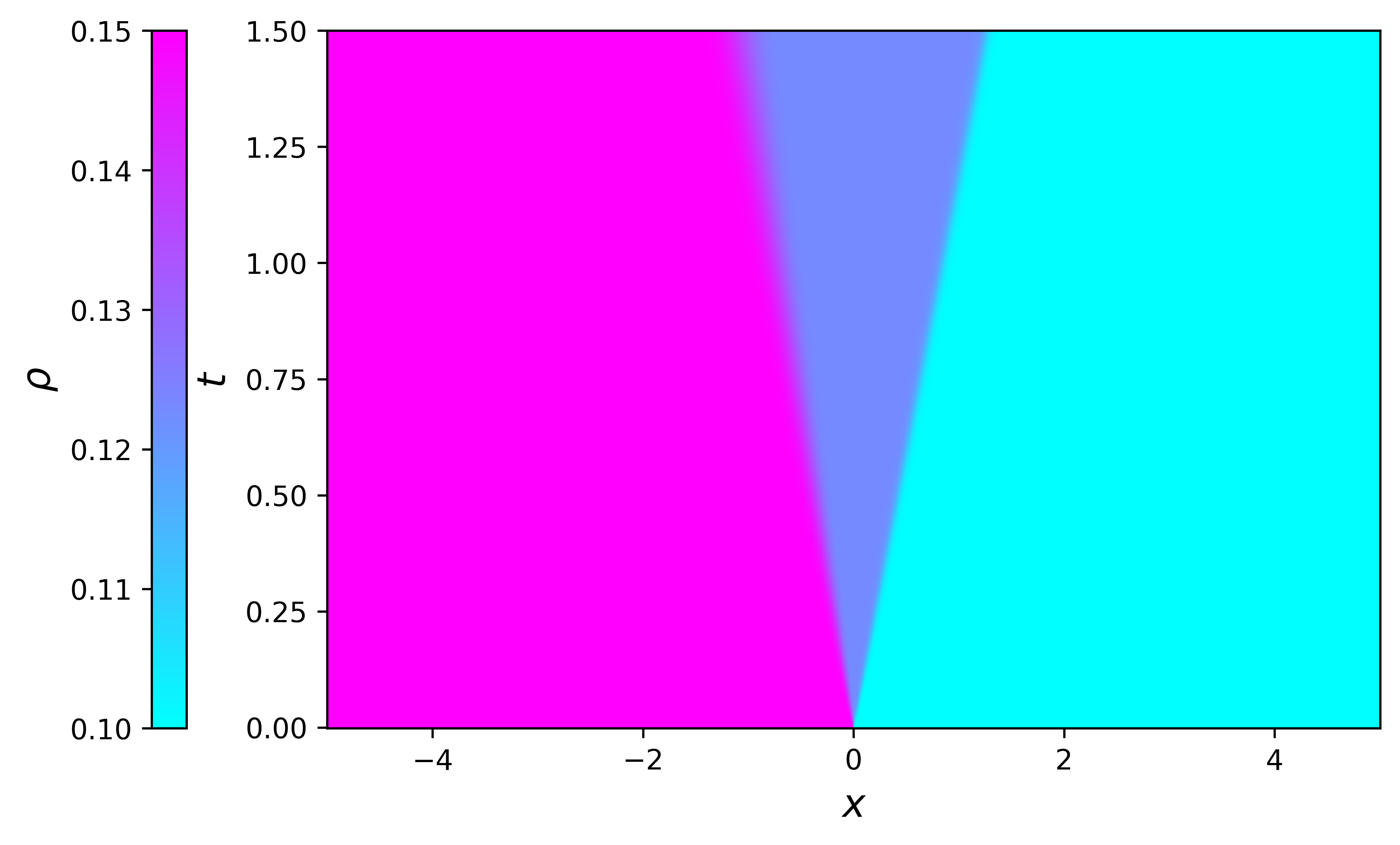}
  \includegraphics[width=0.45\textwidth, height=0.305\textwidth]{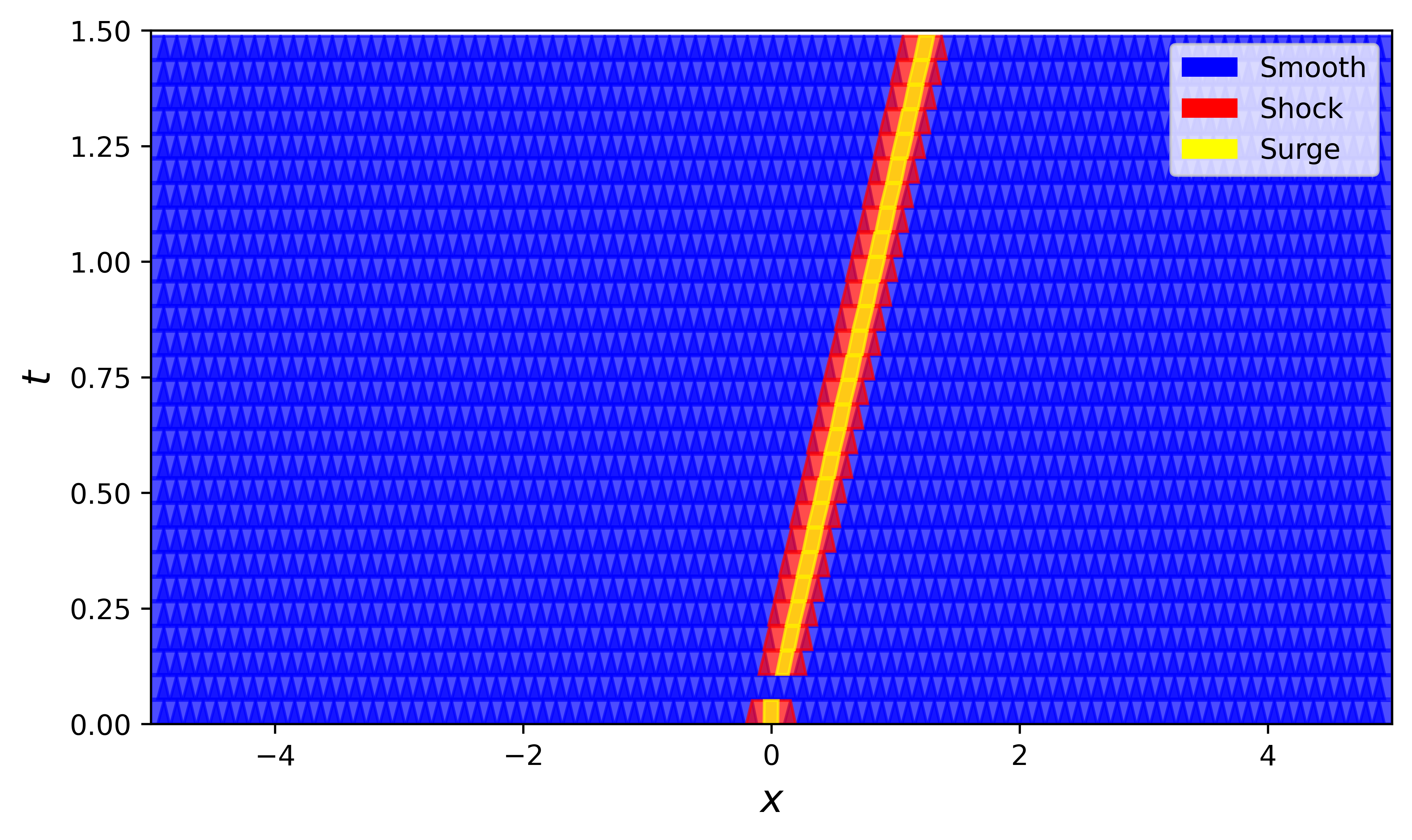}
  \caption{Rarefaction-shock RP for the $p$-system at $L=10$. \textbf{Left:} density $\rho$ in space-time.\quad  \textbf{Right:} surge areas and the trapezoidal decomposition.}
  \label{fig:sod-plot-god}
\end{figure}
Note that Figure~\ref{fig:sod-plot-god} shows that the implementation of Algorithm~\ref{algo:surge-detect} detects a surge in the first meso-timeslab, i.e.~where the waves are not yet fully developed. In the following meso-timeslab no surges are detected, since the distance between the waves is too small. With higher resolution the surge in the initial meso-timeslab cannot be detected anymore, since the waves develop earlier, while the number of meso-timeslabs where no surges are detected shrinks. Thus, the implementation of the algorithm shows expected behavior. 

The results of the error estimator for the rarefaction-shock RP are listed in Table~\ref{tab:sod-god}.
\begin{table}[!htb]
  \begin{center}
    \setlength{\tabcolsep}{1.2ex}
    \begin{tabular}{c|cc|c|cc|cc|cc}
      $L$ & $\varepsilon$ & EoC & $\varepsilon^{\frac13}$ & $\estim_S$ & EoC &  $\estim_G$ & EoC & $L^{\infty}L^1$-err & EoC\\
      \hline
      7 & 0.06268 &  & 0.39724 & 2.85997 &  & 0.84665 &  & 0.00661 & \\
      8 & 0.03134 & 1.00 & 0.31529 & 2.21944 & 0.37 & 0.69120 & 0.29 & 0.00421 & 0.65\\
      9 & 0.01567 & 1.00 & 0.25024 & 1.74306 & 0.35 & 0.55701 & 0.31 & 0.00258 & 0.71\\
      10 & 0.00784 & 1.00 & 0.19862 & 1.39272 & 0.32 & 0.45395 & 0.30 & 0.00153 & 0.76\\
      11 & 0.00392 & 1.00 & 0.15764 & 1.06811 & 0.38 & 0.37343 & 0.28 & 0.00088 & 0.79\\
      12 & 0.00196 & 1.00 & 0.12512 & 0.80363 & 0.41 & 0.30651 & 0.28 & 0.00050 & 0.81\\
    \end{tabular}
  \end{center}
  \caption{Error estimator results for the rarefaction-shock-RP for the $p$-system.}
  \label{tab:sod-god}  
\end{table}
The numerical test for the smooth case shows an experimental order of convergence of the error estimator of $\frac13$ which is the order proven (under the a-posteriori verifiable condition) by Bressan~\cite{bressanPosterioriErrorEstimates2021a}. In the test with discontinuous solution we see a slightly lower order in the smooth part of the estimator which might indicate  that the resolution is not in the asymptotic regime yet. Note that the surge-part of the estimator actually converges faster than the smooth part.

\subsection*{Burgers equation with curved shock position}
\label{sec:burg-equat-with}

In the last test case we consider the Burgers equation with the initial condition
\begin{equation}
  u(0,\, x) =
  \begin{cases}
    10  &\text{ if }\,\, x \leq -4, \\
    -3x - 2 &\text{ if }  -4 < x \leq 0, \\
    -7 &\text{ if }   x > 0. \\
  \end{cases}
\end{equation}
The initial condition is formed by a cut-off linear function to the left of $x=0$, a jump at $x=0$ and a constant state to the right of $x=0$. The solution to this problem consists of  an initially left-traveling shock that switches its travel-direction at around $t=0.3$ (where the linear function steepens to a vertical line) as depicted in Figure~\ref{fig:burgers-plot}. The purpose of this test case is to check the behavior of the $\delta$-width search in Algorithm~\ref{algo:surge-detect}.

The right-hand side of Figure~\ref{fig:burgers-plot} shows the according trapezoidal decomposition. The final time is set to $T=1.0$. In the second meso-timeslab, i.e.~where the shock curve changes its direction the surge width is approximately trice as wide as in other meso-timeslabs. Thus, the implementation of Algorithm~\ref{algo:surge-detect} captures the curved shock position correctly. 
\begin{figure}[!htb]
  \centering
  \includegraphics[width=0.45\textwidth]{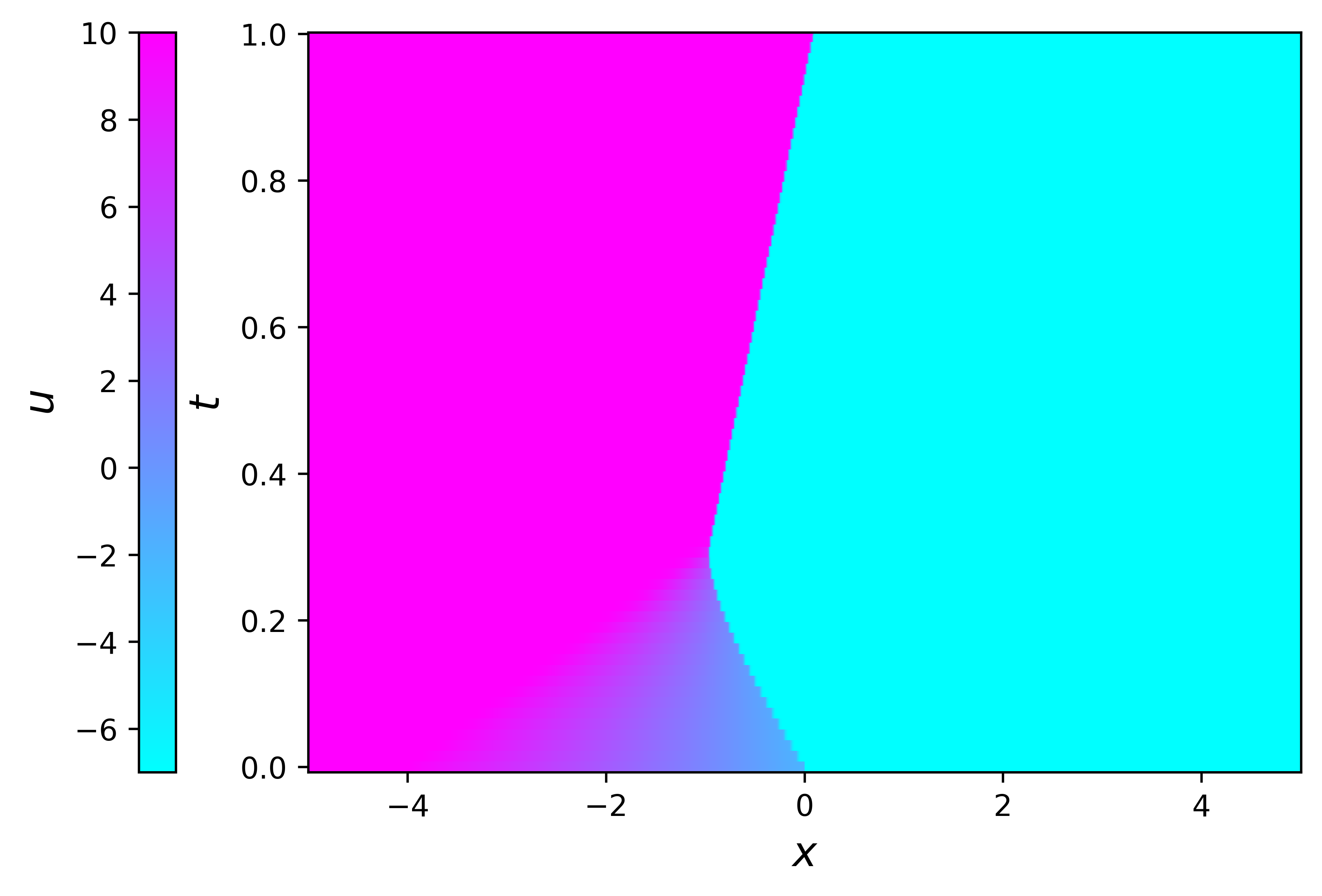}
  \includegraphics[width=0.45\textwidth, height=0.305\textwidth]{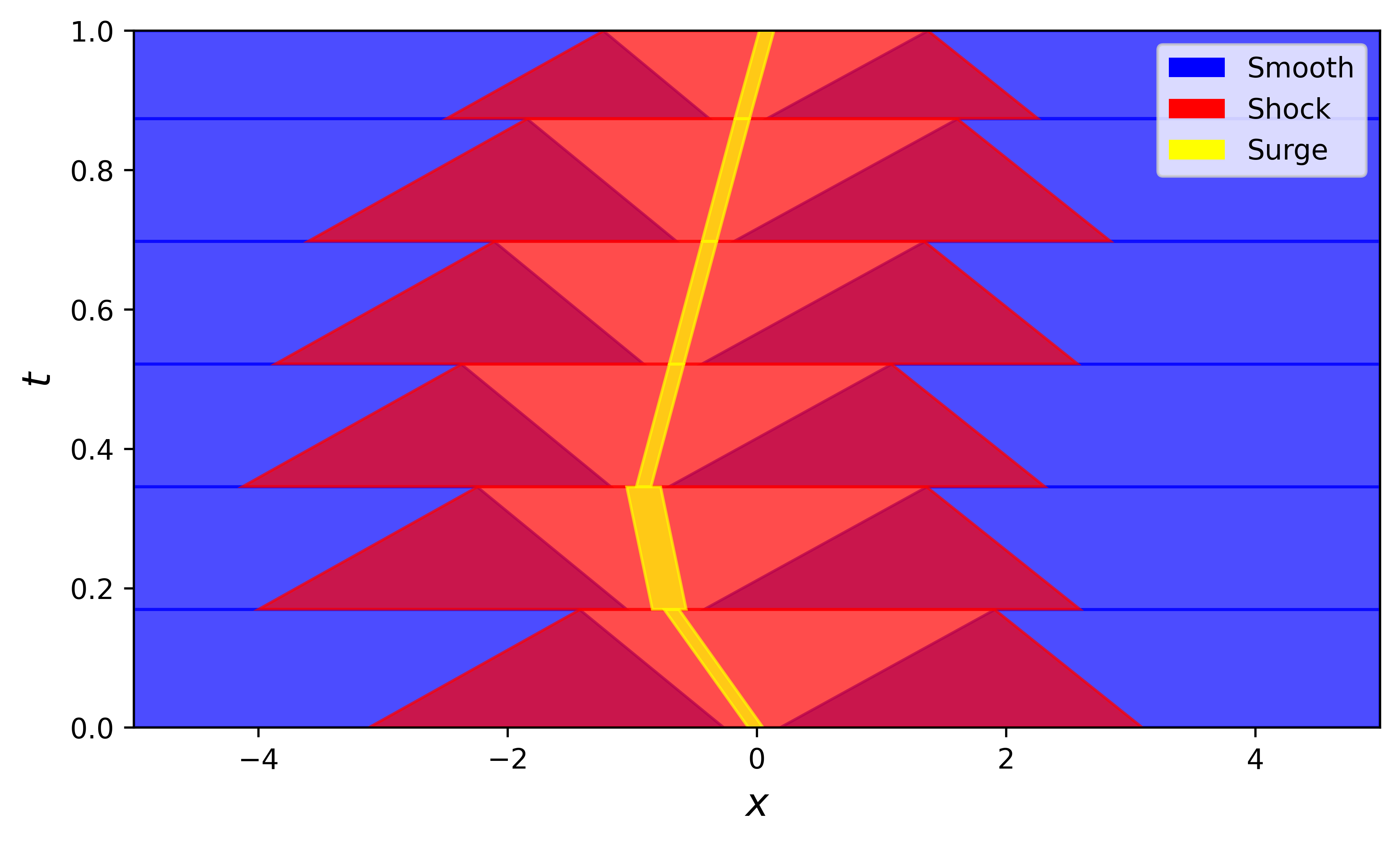}
  \caption{Left: Burgers space-time solution at $L=9$. Right: Corresponding trapezoidal decomposition of the domain.}
  \label{fig:burgers-plot}
\end{figure}

For the computation of the $L^{\infty}L^1$-error we have employed a numerical solution on level $L=14$ instead of the exact solution since the derivation of the latter one is rather cumbersome.  
 \begin{table}[!htb]
  \begin{center}
    \setlength{\tabcolsep}{1.2ex}
    \begin{tabular}{c|cc|c|cc|cc|cc}
      $L$ & $\varepsilon$ & EoC & $\varepsilon^{\frac13}$ & $\estim_S$ & EoC & $\estim_G$ & EoC  & $L^{\infty}L^1$-err & EoC\\
      \hline
      8 & 0.21730 &  & 0.60120 &   6.73095 &  & 10.3601 &  & 0.22878 & \\
      9 & 0.10865 & 1.00 & 0.47717 & 5.61791 & 0.26 & 7.99104 & 0.37 & 0.08885 & 1.36\\
      10 & 0.05432 & 1.00 & 0.37873  & 4.69667 & 0.26 & 5.97280 & 0.42 & 0.04693 & 0.92\\
      11 & 0.02716 & 1.00 & 0.30060  & 4.05053 & 0.21 & 4.90897 & 0.28 & 0.02562 & 0.87\\
      12 & 0.01358 & 1.00 & 0.23859  & 2.83435 & 0.52 & 3.82476& 0.36 & 0.01411 & 0.86\\
    \end{tabular}
  \end{center}
  \caption{ Error estimator results for the Burgers equation with a varying  shock curve direction.}
  \label{tab:burgers}  
\end{table}
Table~\ref{tab:burgers} shows that, except for $L=11$,  the estimated order of convergence for the error estimator corresponds to  the expected asymptotic behavior.


\bibliographystyle{siamplain}
\bibliography{all} 

\begin{thebibliography}{10}

\bibitem{bressanHyperbolicSystemsConservation2000}
{\sc A.~Bressan}, {\em Hyperbolic {{Systems}} of {{Conservation Laws}}: {{The
  One-dimensional Cauchy Problem}}}, Oxford {{Lecture Series}} in
  {{Mathemathics}}, {Oxford University Press}, 2000.

\bibitem{bressanPosterioriErrorEstimates2021a}
{\sc A.~Bressan, M.~T. Chiri, and W.~Shen}, {\em A {{Posteriori Error
  Estimates}} for {{Numerical Solutions}} to {{Hyperbolic Conservation Laws}}},
  Archive for Rational Mechanics and Analysis, 241 (2021), pp.~357--402,
  \url{https://doi.org/10.1007/s00205-021-01653-4}.

\bibitem{chen2022uniqueness}
{\sc G.~Chen, S.~G. Krupa, and A.~F. Vasseur}, {\em Uniqueness and weak-bv
  stability for 2$\times$ 2 conservation laws}, Archive for Rational Mechanics
  and Analysis, 246 (2022), pp.~299--332,
  \url{https://doi.org/10.1007/s00205-022-01813-0}.

\bibitem{dafermosHyperbolicConservationLaws2016}
{\sc C.~M. Dafermos}, {\em Hyperbolic {{Conservation Laws}} in {{Continuum
  Physics}}}, Grundlehren Der {{Mathematischen Wissenschaften}},
  {Springer-Verlag}, fourth~ed., 2016.

\bibitem{dahmenMultiresolutionSchemesConservation2001}
{\sc W.~Dahmen, B.~{Gottschlich{\textendash}M{\"u}ller}, and S.~M{\"u}ller},
  {\em Multiresolution schemes for conservation laws}, Numerische Mathematik,
  88 (2001), pp.~399--443, \url{https://doi.org/10.1007/s211-001-8009-3}.

\bibitem{dednerPosterioriAnalysisFully2016}
{\sc A.~Dedner and J.~Giesselmann}, {\em A {{Posteriori Analysis}} of {{Fully
  Discrete Method}} of {{Lines Discontinuous Galerkin Schemes}} for {{Systems}}
  of {{Conservation Laws}}}, SIAM Journal on Numerical Analysis, 54 (2016),
  pp.~3523--3549, \url{https://doi.org/10.1137/15M1046265}.

\bibitem{gerhardWaveletFreeApproachMultiresolutionBased2021}
{\sc N.~Gerhard, S.~M{\"u}ller, and A.~Sikstel}, {\em A {{Wavelet-Free
  Approach}} for {{Multiresolution-Based Grid Adaptation}} for {{Conservation
  Laws}}}, Communications on Applied Mathematics and Computation, 4 (2021),
  pp.~1--35, \url{https://doi.org/10.1007/s42967-020-00101-6}.

\bibitem{giesselmann2022theory}
{\sc J.~Giesselmann and S.~G. Krupa}, {\em Theory of shifts, shocks, and the
  intimate connections to {$L^2$}-type a posteriori error analysis of
  discontinuous {Galerkin} and other numerical schemes for hyperbolic
  problems}, to appear.

\bibitem{giesselmannPosterioriAnalysisDiscontinuous2015a}
{\sc J.~Giesselmann, C.~Makridakis, and T.~Pryer}, {\em A {{Posteriori
  Analysis}} of {{Discontinuous Galerkin Schemes}} for {{Systems}} of
  {{Hyperbolic Conservation Laws}}}, SIAM Journal on Numerical Analysis, 53
  (2015), pp.~1280--1303, \url{https://doi.org/10.1137/140970999}.

\bibitem{godlewskiHyperbolicSystemsConservation1991}
{\sc E.~Godlewski and P.-A. Raviart}, {\em Hyperbolic Systems of Conservation
  Laws}, {Ellipses}, {Paris}, 1991.

\bibitem{hartmannAdaptiveDiscontinuousGalerkin2002}
{\sc R.~Hartmann and P.~Houston}, {\em Adaptive {{Discontinuous Galerkin Finite
  Element Methods}} for the {{Compressible Euler Equations}}}, Journal of
  Computational Physics, 183 (2002), pp.~508--532,
  \url{https://doi.org/10.1006/jcph.2002.7206}.

\bibitem{jovanovicErrorEstimatesFinite2006}
{\sc V.~Jovanovic and C.~Rohde}, {\em Error {{Estimates}} for {{Finite Volume
  Approximations}} of {{Classical Solutions}} for {{Nonlinear Systems}} of
  {{Hyperbolic Balance Laws}}}, SIAM Journal on Numerical Analysis, 43 (2006),
  pp.~2423--2449, \url{https://doi.org/10.1137/S0036142903438136}.

\bibitem{karni2005local}
{\sc S.~Karni and A.~Kurganov}, {\em Local error analysis for approximate
  solutions of hyperbolic conservation laws}, Advances in Computational
  mathematics, 22 (2005), pp.~79--99,
  \url{https://doi.org/10.1007/s10444-005-7099-8}.

\bibitem{kimAdaptiveVersionGlimm2010}
{\sc H.~Kim, M.~Laforest, and D.~Yoon}, {\em An adaptive version of {{Glimm}}'s
  scheme}, Acta Mathematica Scientia, 30 (2010), pp.~428--446,
  \url{https://doi.org/10.1016/S0252-9602(10)60057-4}.

\bibitem{kronerPosterioriErrorEstimates2000}
{\sc D.~Kr{\"o}ner and M.~Ohlberger}, {\em A posteriori error estimates for
  upwind finite volume schemes for nonlinear conservation laws in multi
  dimensions}, Mathematics of Computation, 69 (2000), pp.~25--39,
  \url{https://doi.org/10.1090/S0025-5718-99-01158-8}.

\bibitem{krupa2019uniqueness}
{\sc S.~G. Krupa and A.~F. Vasseur}, {\em On uniqueness of solutions to
  conservation laws verifying a single entropy condition}, Journal of
  Hyperbolic Differential Equations, 16 (2019), pp.~157--191,
  \url{https://doi.org/10.1142/S0219891619500061}.

\bibitem{kruzkovFirstOrderQuasilinear1970}
{\sc S.~N. Kru{\v z}kov}, {\em First order quasilinear equations in several
  independent variables}, Mathematics of the USSR-Sbornik, 10 (1970), p.~217,
  \url{https://doi.org/10.1070/SM1970v010n02ABEH002156}.

\bibitem{laforestmarcetiennePosterioriErrorEstimate2001}
{\sc M.~{\'E}. Laforest}, {\em A Posteriori Error Estimate for Front-Tracking},
  PhD thesis, State University of New York at Stony Brook, 2001.

\bibitem{laforestPosterioriErrorEstimate2004}
{\sc M.~{\'E}. Laforest}, {\em A {{Posteriori Error Estimate}} for
  {{Front-Tracking}}: {{Systems}} of {{Conservation Laws}}}, SIAM Journal on
  Mathematical Analysis, 35 (2004), pp.~1347--1370,
  \url{https://doi.org/10.1137/S0036141002416870}.

\bibitem{muller2002adaptive}
{\sc S.~M{\"u}ller}, {\em Adaptive multiscale schemes for conservation laws},
  vol.~27, Springer Science \& Business Media, 2002.

\bibitem{nessyahu1994convergence}
{\sc H.~Nessyahu, E.~Tadmor, and T.~Tassa}, {\em The convergence rate of
  godunov type schemes}, SIAM journal on numerical analysis, 31 (1994),
  pp.~1--16, \url{https://doi.org/10.1137/0731001}.

\bibitem{sempliceAdaptiveMeshRefinement2016}
{\sc M.~Semplice, A.~Coco, and G.~Russo}, {\em Adaptive mesh refinement for
  hyperbolic systems based on third-order compact {{WENO}} reconstruction},
  Journal of Scientific Computing, 66 (2016), pp.~692--724,
  \url{https://doi.org/10.1007/s10915-015-0038-z}.

\bibitem{toroRiemannSolversNumerical2009}
{\sc E.~F. Toro}, {\em Riemann {{Solvers}} and {{Numerical Methods}} for
  {{Fluid Dynamics}}: {{A Practical Introduction}}}, {Springer}, third~ed.,
  2009.

\end{thebibliography}

\appendix{}
\newpage
\section{Partitioning of the meso-timeslab}
\label{sec:part-meso-timesl}

\begin{algorithm}[]
  \caption{Partitioning a meso-timeslab  into surge and ``smooth'' trapezoids}\label{algo:dom-part}
  \begin{algorithmic}
    \Function{inb}{$x$} \Comment{{\small In bounds of $\Omega$}}
      \State\Return $\max \{x_{-\frac12}, \min \{x, x_{J+\frac12} \} \}$
    \EndFunction
    \Function{meso-slab-partition}{$\num{\vec{u}}\in {\pwConst([t^n, t^{n+l}]\times\dom)}^m$,  $\varepsilon$, $\shockThresh_0$}
      \State $\mathtt{G}\gets \mathtt{list}(\,\,)$,  $\tau \gets t^{n+l} - t^n$
      \State $\lambda^- \gets \min_{j\in\{0,\ldots, J\}, \nu \in \{n,\ldots, n+l\}} \lambda(\num{\vec{u}}_j^\nu)$, $\lambda^+ \gets \max_{j\in\{0,\ldots, J\}, \nu \in \{n,\ldots, n+l\}} \lambda(\num{\vec{u}}_j^\nu)$
      \State  $\mathtt{S},\, \mathtt{Sosc} \gets \textsc{surge-trpzs}(t^n, \, \Omega,\, \num{\vec{u}}, \, \shockThresh_0, \, \lambda^-,\, \lambda^+)$
      \If {$|\mathtt{surges_{bottom}}| = 0$}  \Comment{{\small i.e. $\num{\vec{u}}$ is ``smooth'' in $[t^n,t^{n+l}]$}}
        \State $\mathtt{G}.append(\mathtt{trpz} (\Omega,\, \Omega))$
        \State\Return{$\texttt{S},\, \texttt{G}$}
      \EndIf
      
      \If{$(x_{-\frac12} < \mathtt{S}_0.S^l.a)$}  \Comment{{\small outermost lhs.~``smooth'' trapezoid}}
        \State $\mathtt{G}.append\left(\mathtt{trpz}\left( [x_{-\frac12} ,\, \mathtt{S}_0.S^l.a - \tau\lambda^{-}-\varepsilon^{\frac23}],\, [x_{-\frac12},\, \mathtt{S}_0.S^l.a]\right)\right)$
      \EndIf
      \For{$j \in 0,\ldots,\mathtt{S}.size()-1$} 
        \State \textbf{if} {$\mathtt{S}_j.S^r.b \geq x_{J+\frac12}$} \textbf{then}  break \textbf{endif} \Comment{{\small i.e. covering the rhs.~domain boundary}}         
        \State $a\gets \mathtt{S}_j.S^r.b$, \quad   $b\gets
        \begin{cases}
          x_{J+\frac12} & \text{\textbf{if}}\quad j = \mathtt{S}.size() - 1,\\
          \mathtt{S}_{j+1}.S^l.a & \text{otherwise}
        \end{cases}$
        \State $\mathtt{G}.append\left(\mathtt{trpz}\left( [a - \tau\lambda^+ - \varepsilon^{\frac23},\, b - \tau\lambda^--\varepsilon^\frac23],\, [a, b]\right)\right) $
      \EndFor
      \State\Return{$\texttt{S},\, \mathtt{Sosc},\, \texttt{G}$} 
    \EndFunction{}
  \end{algorithmic}
\end{algorithm}

\end{document}